\documentclass[12pt]{amsart}
\usepackage[margin=1in,letterpaper,portrait]{geometry}
\usepackage{ytableau}
\ytableausetup{notabloids}
\ytableausetup{centertableaux}
\usepackage{latexsym,amsmath,amssymb,verbatim, tikz}
\usepackage{graphicx}
\usepackage[colorlinks]{hyperref}
\usepackage{mathrsfs}
\usepackage{amsfonts}
\usepackage{amsthm,youngtab}
\usepackage{graphicx}
\usepackage{caption}
\usepackage{enumerate}
\usepackage{tikz-cd}
\usepackage{rotating}
\usepackage[all]{xy}
\usepackage[titletoc, title]{appendix}
\usepackage{amscd}
\usepackage{float}
\hypersetup{
bookmarksnumbered,
pdfstartview={FitH},
breaklinks=true,
linkcolor=blue,
urlcolor=blue,
citecolor=blue,
bookmarksdepth=2
}	
\usepackage{caption}
\usepackage[utf8]{inputenc}
\usepackage[T1]{fontenc}
\usepackage{adjustbox}

 \usetikzlibrary{fit,backgrounds}

\begin{document}

\theoremstyle{plain}
   \newtheorem{thm}{Theorem}[section]
   \newtheorem{prop}[thm]{Proposition}
   \newtheorem{lem}[thm]{Lemma}
   \newtheorem{cor}[thm]{Corollary}
   \newtheorem{conj}[thm]{Conjecture}
   \newtheorem{problem}[thm]{Problem}
\theoremstyle{definition}
   \newtheorem{deftn}[thm]{Definition}
   \newtheorem{example}[thm]{Example}
   \newtheorem{examples}[thm]{Examples}
   \newtheorem{question}[thm]{Question}
   \newtheorem{algorithm}[thm]{Algorithm}
   \newtheorem{rk}[thm]{Remark}
   \newtheorem{obs}[thm]{Observation}

\numberwithin{equation}{section}

\newcommand\comp[2]{\alpha{(#2,#1)}}
\newcommand\ccomp[2]{\alpha^\cyc{(#2,#1)}}

\newcommand{\CC}{{\mathbb {C}}}
\newcommand{\QQ}{{\mathbb {Q}}}
\newcommand{\RR}{{\mathbb {R}}}
\newcommand{\ZZ}{{\mathbb {Z}}}
\newcommand{\NN}{{\mathbb {N}}}

\newcommand{\EE}{{\mathcal{E}}}
\newcommand{\PP}{{\mathcal {P}}}

\newcommand{\Cyl}{\operatorname{Cyl}}
\newcommand{\Tor}{\operatorname{Tor}}

\newcommand{\then}{\Rightarrow}

\newcommand{\Aut}{{\operatorname{Aut}}}
\newcommand{\Des}{{\operatorname{Des}}}
\newcommand{\cDes}{{\operatorname{cDes}}}
\newcommand{\cyc}{{\operatorname{cyc}}}
\newcommand{\cdes}{{\operatorname{cdes}}}
\newcommand{\des}{{\operatorname{des}}}
\newcommand{\ch}{{\operatorname{ch}}}
\newcommand{\shape}{{\operatorname{shape}}}
\newcommand{\SYT}{{\operatorname{SYT}}}
\newcommand{\co}{{\operatorname{co_n}}}
\newcommand{\cc}{{\operatorname{cc}_n}}
\newcommand{\SSYT}{{\rm SSYT}}
\newcommand{\tchi}{{\widetilde{\chi}}}
\newcommand{\td}{{\widetilde{D}}}
\newcommand{\bd}{{\bar{D}}}
\newcommand{\klr}{{\rm KLR}}

\newcommand{\Comp}{{\operatorname{Comp}}}
\newcommand{\cComp}{{\operatorname{cComp}}}
\newcommand{\ccon}{{\operatorname{cyc-comp}(n)}}

\newcommand{\one}{{\mathbf{1}}}

\newcommand{\cs}{{\tilde{s}}}

\newcommand{\xx}{{\mathbf{x}}}
\newcommand{\ttt}{{\mathbf{t}}}
\newcommand{\symm}{{\mathfrak{S}}}

\def\AA{\mathbb{A}}

\newcommand{\AAA}{{\mathcal{A}}}
\newcommand{\EEE}{{\mathcal{E}}}
\newcommand{\CCC}{{\mathcal{C}}}
\newcommand{\LLL}{{\mathcal{L}}}
\newcommand{\MMM}{{\mathcal{M}}}
\newcommand{\OOO}{{\mathcal{O}}}
\newcommand{\TTT}{{\mathcal{T}}}
\newcommand{\shihao}{\fontsize{3.25pt}{\baselineskip}\selectfont}

\newcommand{\stirling}[2]{\left\{\begin{matrix}{#1}\\ {#2}\end{matrix}\right\}}

\DeclareRobustCommand{\stirlingI}{\genfrac \langle \rangle {0pt}{}}
\newcommand{\GaussBinomial}[2]{\left[\begin{matrix}{#1}\\ {#2}\end{matrix}\right]}

\newcommand{\SteinbergTorus}{{\widetilde{\bdelta}}}

\newcommand{\wcDes}{\cDes_*}
\newcommand{\wcdes}{\cdes_*}

\newcommand{\AffDes}{{\operatorname{cDes}}}

\newcommand{\udots}{\reflectbox{$\ddots$}}

\newcommand{\bT}{{\mathbf T}}
\newcommand{\bA}{{\mathbf A}}
\newcommand{\tensor}{\otimes}
\newcommand{\bPsi}{\boldsymbol\Psi}
\newcommand{\kk}{\Bbbk}
\newcommand{\lie}[1]{{\mathfrak{#1}}}

\newcommand\scalemath[2]{\scalebox{#1}{\mbox{\ensuremath{\displaystyle #2}}}}

\newlength{\mysizetiny}
\setlength{\mysizetiny}{0.3em}
\newlength{\mysizesmall}
\setlength{\mysizesmall}{0.8em}
\newlength{\mysize}
\setlength{\mysize}{1.3em}
\newlength{\mysizelarge}
\setlength{\mysizelarge}{2em}

\newcommand{\IZ}{\widehat{I}_{\mathbb{Z}}}
\newcommand{\Ixi}{\widehat{I}_{\leq \xi}}
\newcommand{\YZ}{\mathcal{Y}_{\mathbb{Z}}}
\newcommand{\tdxi}{\widetilde{D}_{\xi}}
\newcommand{\barD}{\overline{D}}
\newcommand{\barDstar}{\overline{D}^{*}}
\newcommand{\xik}{\tau^{-1}x_k^{\mathbf{i}}}
\newcommand{\xkop}{\tau^{-1}x_k^{op}}
\newcommand{\betakop}{\beta_k^{op}}
\newcommand{\CN}{\mathbf{k}[\mathbf{N}]}
\newcommand{\si}{\mathcal{S}^{\mathbf{i}}}
\newcommand{\ihat}{\widehat{\mathbf{i}}}
\newcommand{\HQ}{\mathcal{H}_Q}
\newcommand{\AQ}{\mathcal{A}_Q}
\newcommand{\MQ}{\mathcal{M}_Q}
\newcommand{\tMQ}{\widetilde{\mathcal{M}_Q}}
\newcommand{\ttMQ}{\widetilde{\widetilde{\mathcal{M}_Q}}}
\newcommand{\DQ}{\mathcal{D}_Q}
\newcommand{\IQ}{\mathcal{I}_Q}
\newcommand{\KQ}{\mathcal{K}_Q}
\newcommand{\TQ}{\mathcal{T}_Q}
\newcommand{\tCQ}{\widetilde{\mathcal{C}_Q}}
\newcommand{\ZQ}{\mathbb{Z}Q}
\newcommand{\ZQz}{(\mathbb{Z}Q)_0}
\newcommand{\CZ}{\mathcal{C}_{\mathbb{Z}}}
\newcommand{\CxiZ}{\mathcal{C}_{\mathbb{Z}}^{\leq \xi}}
\newcommand{\Cxi}{\mathcal{C}_{\xi}}
\newcommand{\YxiZ}{\mathcal{Y}_{\mathbb{Z}}^{\leq \xi}}
\newcommand{\CxiZstar}{\mathcal{C}_{\xi^{*}}^{\mathbb{Z}}}
\newcommand{\barCQ}{\overline{\mathcal{C}_Q}}
\newcommand{\tdKxi}{\widetilde{K}_{\xi}}
\newcommand{\tdKxistar}{\widetilde{K}_{\xi^{*}}}
\newcommand{\DbRepQ}{\mathcal{D}^b(\mathrm{mod}\,\mathbf{k}Q)}
\newcommand{\xistar}{\xi^{*}}
\newcommand{\bkZQ}{\overline{k(\mathbb{Z}Q)}}
\newcommand{\xbar}{\overline{x}}
\newcommand{\ybar}{\overline{y}}
\newcommand{\tH}{\widetilde{\mathcal{H}}}
\newcommand{\ObQ}{\mathrm{Ob}_Q}
\newcommand{\ObCQ}{\mathrm{Ob} \left( \mathcal{C}_Q \right)}
\newcommand{\tHi}{\widetilde{H}_i}
\newcommand{\tWi}{\widetilde{W}_i}
\newcommand{\tHj}{\widetilde{H}_j}
 \newcommand{\bYx}{Y_{\bullet}(x)}
  \newcommand{\bFx}{F_{\bullet}(x)}
 \newcommand{\HQG}{\mathcal{C}_Q^{\Gamma}}
 \newcommand{\IG}{\mathcal{J}^{\Gamma}}
 \newcommand{\KQG}{\mathcal{K}_Q^{\Gamma}}
\newcommand{\mbeta}{\mathfrak{m}_{\beta}}
\newcommand{\bdelta}{\boldsymbol{\delta}}
\newcommand{\bomega}{\boldsymbol{\omega}}
\newcommand{\Qbeta}{Q_{\beta}}
   \newcommand{\Ybeta}{Y^{+}[\beta]}
 \newcommand{\Ymbeta}{Y^{-}[\beta]}
   \newcommand{\Cbeta}{C_{\bullet}^{\beta}}
   \newcommand{\Zbeta}{Z_{\beta}}
   \newcommand{\CXhd}{C_{\bullet}[(X,h),d]}
\newcommand{\JbetaL}{\mathcal{J}_{Q,\beta}^L}
\newcommand{\JbetaR}{\mathcal{J}_{Q,\beta}^R}
\newcommand{\CXh}{C_{\bullet}(X,h)}
\newcommand{\barpi}{\overline{\pi}}
\newcommand{\tRQ}{\widetilde{\mathcal{R_Q}}}
\newcommand{\RQ}{\mathcal{R}_Q}
\newcommand{\HZQ}{\mathcal{H}_{\mathbb{Z}Q}}

    \newcommand{\odi}{\overline{\mathbf{d}_i}}
    \newcommand{\udi}{\underline{\mathbf{d}_i}}

\title[Higher homological algebra and $q$-characters]{A higher homological approach to the $q$-characters of representations of quantum affine algebras}
\author{\'Elie Casbi}
\address{Oskar-Morgenstern Platz 1, 1090 Wien, Österreich}
\email{elie.casbi@univie.ac.at}
\date{}

\begin{abstract}
  For any acyclic quiver $Q$ without multiple edges, we construct a monoidal category $\RQ$ whose indecomposable objects are tensor products (over the base field) of finite-dimensional modules over the path algebra of $Q$.  We show the existence and uniqueness up to homotopy of certain distinguished  chain complexes satisfying good homological properties (higher almost split complexes) preserved under tensoring by objects in $\RQ$. As a crucial ingredient for this construction, we establish  the existence of a family of complete exceptional sequences in $\mathrm{mod}\,\mathbf{k}Q$ satisfying many good properties, which we believe might be of independent interest.  We then prove that when $Q$ admits a height function, the Euler characteristics of (the images under certain additive functor of) these complexes coincide with the truncated $q$-characters of the standard modules in  Hernandez-Leclerc's category $\mathcal{C}^{(1)}$. Applying our results to the case where the underlying graph of $Q$ is a Dynkin diagram of type $A_n, n \geq 1$, we also interpret the cluster characters of all cluster variables in the finite type cluster algebra $\mathcal{A}_Q$ as Euler characteristics of certain chain complexes in $\RQ$.  
\end{abstract}

\maketitle

\section{Introduction}

 \subsection{Background and motivations}

Let $U_q(\mathfrak{g})$  denote the quantum affine algebra associated to  an affine  Kac-Moody algebra $\mathfrak{g}$. The understanding of simple $U_q(\mathfrak{g})$-modules  is a non-trivial problem that has led to many important advances over the past decades.  Chari-Pressley \cite{CP91} gave a classification of all simple finite-dimensional $U_q(\mathfrak{g})$-modules (up to isomorphism) in terms of dominant monomials. On the other hand the attempts to describe their $q$-characters as defined by  Frenkel-Reshetikhin's  \cite{FR1} led to the discovery of deep connections with the theory of  Fomin-Zelevinsky's cluster algebras  \cite{FZ1}, via the notion of monoidal categorification introduced by Hernandez-Leclerc \cite{HL10}. More precisely, when $\mathfrak{g}$ is an affine Kac-Moody algebra, certain discrete versions of the category of finite-dimensional $U_q(\mathfrak{g})$-modules called  Hernandez-Leclerc categories have been shown to categorify  various interesting cluster structures in such a way that the  set of cluster monomials is contained into  the basis of isomorphism classes of simple objects  in the Grothendieck ring of  $\mathcal{C}_{\mathfrak{g}}$
\cite{HL10,HL16,KKOP1,KKOP2,BC}. This allows to relate certain invariants from cluster theory  ($g$-vectors, $F$-polynomials) to representation-theoretic features (loop weights, $q$-characters).  Recently, monoidal categorifications have also  been obtained using representations of shifted quantum affine algebras \cite{GHL}. In a different direction,  when $\mathfrak{g}$ is an arbitrary symmetrizable Kac-Moody algebra, the $q$-characters of certain modules over a Borel subalgebra of $U_q(\mathfrak{g})$ have been studied by Hernandez-Jimbo \cite{HerJim}, and Negut \cite{Negut} recently provided uniform formulas using the theory of shuffle algebras.

 Among the irreducible representations of $U_q(\mathfrak{g})$, the fundamental representations play a special role as they generate $\mathcal{C}_{\mathfrak{g}}$ as a monoidal category. More algebraically, the isomorphism classes of suitably ordered tensor products of fundamental representations (called standard modules) form a basis of the Grothendieck group of $\mathcal{C}_{\mathfrak{g}}$. We will respectively denote by $L(\mathfrak{m})$ and $\Delta(\mathfrak{m})$ the simple and the standard $U_q(\mathfrak{g})$-module corresponding to a dominant monomial $\mathfrak{m}$.
 When $\mathfrak{g}$ is an affine simply-laced Kac-Moody algebra, the $q$-characters of fundamental representations admit interesting combinatorial descriptions as well as non-trivial geometric interpretations  \cite{Naka01,HL16}. For more general Kac-Moody algebras, the fundamental representations are not anymore finite-dimensional, but still admit interesting combinatorial descriptions in certain cases.

 The main motivation for the construction proposed in the present work stems from our desire to provide a conceptual interpretation of some remarkable rational identities obtained in our recent joint work with Jian-Rong Li \cite{CasbiLi2}. 
 For a   Kac-Moody algebra $\mathfrak{g} = \widehat{\mathfrak{g}_0}$ such that $\mathfrak{g}_0$ is simply-laced,  and for any  orientation $Q$ of the Dynkin diagram of $\mathfrak{g}_0$, we defined (originally in \cite{CasbiLi}, and then under its definite form in \cite{CasbiLi2})  an algebra homomorphism $\td_Q : \YZ \longrightarrow \mathbf{k}(\mathfrak{t}^*)$ which is intimately related to the map $\barD$ defined in \cite{BKK} by Baumann-Kamnitzer-Knutson  in their study of the equivariant homology of Mirkovi\'c-Vilonen cycles. Here $\YZ$ denotes the ring of Laurent polynomials containing the $q$-characters of all objects in the Hernandez-Leclerc's $\CZ$, while $\mathfrak{t}$ stands for a Cartan subalgebra of $\mathfrak{g}_0$.   In \cite{CasbiLi2}, we used the map $\td_Q$ to exhibit new families of non-trivial rational identities out of the representation theory of quantum affine algebras. More precisely, we showed that for any standard module $M$ in $\CZ$, we have 
 $$ \td_Q \left( \chi_q(M) \right) = 0 .  $$
  The starting point of the present work consists in viewing these identities as the vanishing of certain alternate sums of dimension vectors (after getting rid of the denominators). In other words, to each object in $\CZ$ should correspond a chain complex satisfying some kind of exactness property, and whose objects should keep track (in a bijective way) of the Laurent monomials appearing in the $q$-character of $M$. 
  Before presenting our results, we briefly describe the  categories that will serve as underlying framework for the constructions proposed in this paper.
  
 \subsection{The hammock category}

 Given a finite acyclic quiver $Q$, let us denote by $\ZQ$ the repetition quiver of $Q$, $\mathrm{mod}\,\mathbf{k}Q$ the category of finite-dimensional modules over the path algebra of $Q$, and by $\DbRepQ$ its bounded derived category. For each $x \in \ZQz$ (viewed as the isomorphism class of an indecomposable object in $\DbRepQ$ via Happel's theorem), we consider a  pair $Y(x) := (H(x),h_x)$ where $H(x)$ is a  collection (or multiset) consisting of all isomorphism classes of indecomposable objects in $\DbRepQ$ such that $\mathrm{Hom}_{\DbRepQ}(x,y) \neq 0$, and $h_x : \ZQz \rightarrow \mathbb{Z}$ is an integer-valued function on the repetition quiver of $Q$ satisfying a certain property called \textbf{quasi-additivity} (cf. Definition~\ref{def : quasi-additive functions}). Taking unions of multisets and sums of quasi-additive functions, we obtain a family of pairs $(X,h)$ that we call \textbf{dominant objects}. We then consider a certain operation on such pairs $(X,h)$, called \textbf{Serre tilting}, which consists in replacing finitely many elements of $X$ by their image under the standard Serre functor of $\DbRepQ$, as well as modifying  $h$ in a suitable way. The pairs obtained this way form the class of objects of a monoidal category $\HQ$.  Given two multisets $X = \{\tau^{-1}x_k, k \in  K\}$ and $Y = \{y_l , l \in L\}$ and two quasi-additive  functions $f$ and $g$, the space of morphisms between $(X,f)$ and $(Y,g)$ in $\HQ$ is given by
 \begin{equation}
     \bigoplus_{  \substack{ \sigma \in \mathrm{Bij}(K,L) \\  \sharp \mathrm{Supp}_{X,Y}(\sigma) < \infty }} \bigotimes_{k \in \mathrm{Supp}_{X,Y}(\sigma)} \mathrm{Hom}_{\DbRepQ}(\tau^{-1}x_k , y_{\sigma(k)}) 
 \end{equation}
where $\mathrm{Bij}(K,L)$ denotes the set of bijections from $K$ to $L$ and $\mathrm{Supp}_{X,Y}(\sigma) := \{k \in K \mid y_{\sigma(k)} \neq \tau^{-1}x_k \}$ for any $\sigma \in \mathrm{Bij}(K,L)$.
  We call $\HQ$ the  \textbf{hammock category}, as it is essentially governed by the supports of $\mathrm{Hom}(x,-)$ for $x \in \DbRepQ$, which have been extensively studied via the hammock functions  considered by Gabriel, Brenner and Ringel \cite{Gabriel, Brenner,Ringel}. Eventually, the correspondence at the combinatorial level between our construction and the theory of $q$-characters will hinge on the hammock category: for instance, the Grothendieck group of $\HQ$ can be essentially viewed as a subring of the ring $\YZ$ containing the $q$-characters of objects in the Hernandez-Leclerc's category $\CZ$.  The classes of the dominant pairs in $\HQ$ can then be identified dominant monomials, while Serre tiltings provide an analogue of Chari's braid group action on $\YZ$ \cite{Chari02}. 
 
   \subsection{Higher almost split complexes}

   One of the main desired outcomes motivating this work consists in constructing well-behaved chain complexes in $\HQ$ whose objects shall keep track of the $q$-characters of distinguished objects in Hernandez-Leclerc's categories, such as simple or standard objects. A sensible way to tackle this question is to construct these complexes inductively by some iterated mapping cone procedure. Classically, such method works well provided the complexes constructed at each step satisfy good homological properties. Important instances of such properties  include the notion of higher exact or almost split complexes, higher Auslander-Reiten sequences, that have been the subject to numerous major works in the past few years (for example \cite{Iyama07,Jasso16,GKO13}). 
   In order to establish the existence of chain complexes satisfying similar homological properties, it will turn out being necessary to avoid working directly in $\HQ$, but rather in  the (abelian) category 
   $$ \tRQ := \bigoplus_{n \geq 0} \mathrm{mod}\, \mathbf{k}Q^{\otimes n} . $$
   The indecomposable objects in $\tRQ$ are tensor products (over the base field $\mathbf{k}$) of indecomposable $\mathbf{k}Q$-modules. 
  Note that certain special families of higher Auslander-Reiten sequences living in $\tRQ$  have been considered in \cite{Pasquali17}. However, they do not seem to be  appropriate for our needs. 
 In a sense, $\tRQ$ may feel more convenient than the hammock category, as it is  nothing but a concrete category of modules over (tensor product of) path algebras, and also its monoidal structure is more refined. However, several significant obstructions appear:  
 \begin{itemize}
     \item naively replacing objects of $\HQ$ by tensor products of $\mathbf{k}Q$-modules becomes problematic outside Dynkin types, as the hammocks then become infinite.  
     \item the chain complexes that we wish to construct might involve morphisms in $\DbRepQ$. 
     \item the class of morphisms in $\tRQ$ is gigantic, which makes it a priori difficult to construct exact or almost split sequences. 
 \end{itemize}
 The way to circumvent these obstructions is to work inside a well-chosen monoidal (but not full !) subcategory $\RQ$ of $\tRQ$. Note that in the literature, subcategories of great interest are provided by cluster tilting subcategories. However at this moment it does not seem clear how $\RQ$ can fit into this framework. Rather, we construct  $\RQ$ somewhat explicitly, by proving the existence of certain distinguished classes of objects and morphisms in $\tRQ$ that satisfy all the properties necessary for our needs. We achieve this using the theory of exceptional sequences in $\mathrm{mod}\,\mathbf{k}Q$.

 \subsection{A special family of exceptional sequences in $\mathrm{mod}\,\mathbf{k}Q$}

  Exceptional sequences were introduced by Crawley-Boevey \cite{CB92}, generalizing Rudakov's exceptional bundles  \cite{Ru90}. They have been extensively studied and generalized since then (e.g.  \cite{Ringel93,Ringel98,IguTod17,BuanMarsh21})  and have become an important notion in particular in the theory of additive categorifications of cluster algebras. Crawley-Boevey's braid group action on exceptional sequences provides an inductive way of constructing  indecomposable exceptional (i.e. without self extensions) objects in $\mathrm{mod}\,\,\mathbf{k}Q$,  which is particularly useful in wild types where these admit no combinatorial classification in general. In the present paper, we use exceptional sequences to select in a non-trivial way certain \textit{finite} subcollections of the  hammocks $H(P_i) , i \in Q_0$ (here $P_i$ denotes the indecomposable projective $\mathbf{k}Q$-modules). This is done by proving the existence of a remarkable family of exceptional sequences in $\mathrm{mod}\,\mathbf{k}Q$ satisfying strong properties. We refer to Theorem~\ref{thm : exceptional sequences} below for a detailed statement.

   \subsection{Purely exact complexes in $\RQ$}

  As mentioned above, we are going to construct certain chain complexes with objects in $\RQ$ satisfying some higher exactness properties. However, we will need a slightly stronger property, called \textbf{pure exactness}, essentially guaranteeing some compatibility with the monoidal structure of $\RQ$. More precisely, a chain complex $C_{\bullet}$ will  be called purely $\RQ$-exact if $X \otimes C_{\bullet}$ is $\RQ$-exact for every object $X$ in $\RQ$.  The following is the first main result of this paper. 
  
  \begin{thm} \label{thm : thm1 intro}
  Let $Q$ be an acyclic quiver without multiple edges. Then for every indecomposable object $M$ in $\RQ$, there exists a (unique up to homotopy) purely $\RQ$-exact almost split complex $C_{\bullet}(M)$ such that $C_0 \simeq M$. 
  \end{thm}

 \subsection{Truncated $q$-characters as Euler characteristics}

 Given a chain complex  $C_{\bullet} = \cdots \rightarrow C_0 \rightarrow C_1 \rightarrow \cdots $   with objects in $\RQ$ (resp. $\HQ$), we will denote by $\chi(C_{\bullet})$ its Euler characteristics, which is the element of the split Grothendieck group $K_0(\RQ)$ (resp. $K_0(\HQ)$) given by 
 $$ \chi \left( C_{\bullet} \right) := \sum_{n \in \mathbb{Z}} (-1)^n [C_n]. $$
 We can now state the second main result of this paper, which provides an interpretation of the (truncated) $q$-characters of the standard modules $\Delta(\mathfrak{m})$ in Hernandez-Leclerc's category $\Cxi^{(1)}$ as Euler characteristics of the (images under $\mathcal{D}_Q$ of the) chain complexes given by Theorem~\ref{thm : thm1 intro}.  

  \begin{thm} \label{thm : thm2 intro}
  Assume  that $Q$ admits a height function. Let $M$ be an object in $\RQ$ and let $\mathfrak{m}$ denote the dominant monomial in $\YZ$ given by the class of $\mathcal{D}_Q(M)$ in $K_0(\HQ)$. Then we have 
      $$ \chi \left(  \mathcal{D}_Q \left( C_{\bullet}(M)  \right) \right)_{\mid F_i := -1} = \tchi_q \left( \Delta(\mathfrak{m}) \right) .  $$
  \end{thm}
The condition that $Q$ admits a height function is always satisfied if  the underlying graph of $Q$ is a tree (in particular if $Q$ is a Dynkin quiver or a tame quiver of type $D_n^{(1)}$ or $E_n^{(1)}$) and is also satisfied for certain quivers of tame type $A_{2n+1}^{(1)}, n \geq 1$. For other kinds of quivers, we believe the Laurent polynomials provided by the left hand side in the identity of Theorem~\ref{thm : thm2 intro} may be of independent interest. 

 The strong homological properties satisfied by the chain complexes $C_{\bullet}(M)$ make them a powerful tool to exhibit interesting objects and morphisms in the bounded homotopy category of $\RQ$. In particular, one may wonder whether this approach may allow to also realize the truncated $q$-characters of (at least some) \textit{simple} $U_q(\mathfrak{g})$-modules as Euler characteristics of certain well-behaved chain complexes in $\RQ$. The  next result illustrates that this is indeed the case. We focus on  the case where $\mathfrak{g} = \widehat{\mathfrak{sl}_{n+1}}$ and $Q$ is an arbitrary orientation of the Dynkin diagram of $\mathfrak{sl}_{n+1}$, leaving more general situations for future investigation.

  \begin{thm} \label{thm : thm3 intro}
    Assume $Q$ is a Dynkin quiver of type $A_n , n \geq 1$. Then for each positive root $\beta \in \Delta_+$, there exists a chain complex $\Cbeta$ in $\RQ$ such that 
    $$ \chi \left( \mathcal{D}_Q(\Cbeta) \right)_{\mid F_i := -1} = \tchi_q(L[\beta]) $$
    where $L[\beta]$ is the (unique up to isomorphism) simple module in Hernandez-Leclerc's category $\mathcal{C}^{(1)}$ categorifying the cluster variable of denominator vector $\beta$ in $\AQ$. 
  \end{thm}

 In view of the monoidal categorification results by  Brito-Chari \cite{BC} and Kashiwara-Kim-Oh-Park \cite{KKOP1}, all simple modules in $\mathcal{C}^{(1)}$ are cluster monomials and thus simple truncated $q$-characters can be essentially identified with cluster characters. Thus the proof of Theorem~\ref{thm : thm3 intro}  will be carried out purely in the language of cluster algebras. 

  \subsection{Perspectives} 

 The framework proposed in this paper allows to construct  new instances of complexes satisfying higher exactness (or almost splitness) properties, in a compatible way with the monoidal structure of the underlying category.  Thus, our construction may allow to produce large families of new non-trivial examples of $n$-exact sequences (potentially including the case $n=0$  recently developed in \cite{Gulisz25}).
 The fact that their Euler characteristics coincide with certain distinguished (standard or simple) $q$-characters hence builds a new bridge between the representation theory of quantum affine algebras on the one hand and the  theory of $n$-exact complexes on the other hand.

 We  would also like to outline some upshot from a cluster-theoretic perspective. Indeed, while quantum affine algebras are involved in the theory of monoidal categorifications of cluster algebras, all the key features making possible the construction of our chain complexes are controlled by objects typical from the theory of additive categorifications of cluster algebras, above all exceptional sequences. Thus we believe our approach may possibly yield to a better understanding  of the  connections between these two theories at a categorical level. 

  Finally, though the cluster-related result of this paper (Theorem~\ref{thm : thm3 intro}) focus on the type $A$ case, we expect the framework developed in this paper to provide powerful tools to realize more general cluster characters as Euler characteristics, raising the question of the potential existence of monoidal categorifications of other cluster structures than the ones coming from Lie theory.

   \subsection{Outline of the paper}
    
The paper is organized as follows.  Sections~\ref{sec : background} and~\ref{sec : reminders QAAs} are respectively devoted to the necessary recollections on Auslander-Reiten theory and Hernandez-Leclerc's categories of modules over quantum affine algebras. Section~\ref{sec : hammock category} introduces the hammock category and relates it to the theory of $q$-characters. In Section~\ref{sec : exceptional sequences}, we present a result on complete exceptional sequences in  $\mathrm{mod}\,\mathbf{k}Q$ which is the crucial technical feature upon which our constructions hinge. In Section~\ref{sec : the category RQ} we define the category $\RQ$ which will serve as underlying category for the main results of this paper. Section~\ref{sec : exact complexes} is devoted to the main homological definitions  as well as the proofs of Theorems~\ref{thm : thm1 intro} and~\ref{thm : thm2 intro}. Finally, in Section~\ref{sec : Euler characteristics and qcharacters} we prove Theorem~\ref{thm : thm3 intro} after a brief recollection on cluster theory.

 \subsection*{Acknowledgements}
  I would like to express  all my gratitude to Gordana Todorov and Bernhard Keller for patiently providing me their insights and answering my questions about Auslander-Reiten theory. I also thank Masaki Kashiwara, Pavel Etingof,  Bernard Leclerc, David Hernandez,  Daniel Labardini-Fragoso, Jian-Rong Li and Keyu Wang for fruitful discussions. Finally, I warmly thank Vitor Gulisz for making very useful suggestions on early versions of this work. This project was funded by the Austrian Science Fund FWF 10.55776/ESP3162324.

 \section{Auslander-Reiten theory} \label{sec : background}

  \subsection{Auslander-Reiten theoretic background}
 \label{sec : reminders on AR}

  Let $Q$ be a (finite) acyclic quiver multiple edges. Let us denote by $I := Q_0$ its set of vertices. We denote by $\ZQ$ the repetition quiver of $Q$, whose set of vertices is $\ZQz  := \{(i,m) , i \in I , m \in \mathbb{Z}\}$ and whose set of arrows is given by 
$$ \forall \alpha := i \rightarrow j ,  \quad (i,m) \rightarrow (j,m) \enspace \text{and} \enspace (j,m) \rightarrow (i,m+1) .  $$
 By definition, a section of $Q$ in $\ZQ$ refers to any subset of $\ZQ$ of the form $\{(i,m) , i \in I\}$ for a  In particular, given any vertex $x \in \ZQz$, there is a unique section of $Q$ in $\ZQ$ containing $x$. We shall denote it by $Q_x$. given $m \in \mathbb{Z}$. 
  We say that $Q$ admits a height function if there exists a function $\xi : I \rightarrow \mathbb{Z}$ such that $\xi(j) = \xi(i)-1$ if there is an arrow $i \rightarrow j$ in $Q$ (such a function will then be called a height function adapted to $Q$). If $Q$ admits a height function, then the set of vertices and arrows of $\ZQ$ can be conveniently respectively described  as 
  $\ZQz := \{(i,p) , p \in i + 2 \mathbb{Z}\}$ and 
  $$ \forall \alpha  := i \rightarrow j,  \quad (j,p) \rightarrow (i,p+1) \enspace \text{and} \enspace (i,p) \rightarrow (j,p+1) .  $$
   A section of $Q$ in $\ZQ$  then corresponds to any subset of $\ZQ$ of the form $\{(i,p_i) , i \in I\}$ such that $p_{s(\alpha)} - p_{t(\alpha)} = 1$ for every arrow $\alpha$ in $Q$.
   We will denote by $\tau$ the Auslander-Reiten translate on $\ZQ$ which is given by $\tau (i,p) := (i,p-2)$.
   In all what follows we will be working over an algebraically closed field $\mathbf{k}$ of characteristic zero. We denote by $\mathrm{mod}\,\mathbf{k}Q$ the category of finite dimensional representations of $Q$ over $\mathbf{k}$ (equivalently, finite-dimensional modules over the path algebra $\mathbf{k}Q$ of $Q$) and by $\DbRepQ$ the bounded derived category of $\mathrm{mod}\,\mathbf{k}Q$. It follows from Happel's theorem (Proposition 4.6 in \cite{Happel}) that there is an injective set-theoretic map $j_Q$ from $\ZQz$ to the set of isomorphism classes of indecomposable objects in $\DbRepQ$. This injection is moreover a bijection in the case where $Q$ is a Dynkin quiver, i.e. an orientation of the Dynkin diagram of a simple complex Lie algebra of type $A_n , n \geq 1 , D_n , n \geq 4$ or $E_n , n = 6,7,8$.
   We denote by $\Sigma$ the standard shift in $\DbRepQ$ and by $\tau$  the Auslander-Reiten  translate functor in $\DbRepQ$. It follows from the Auslander-Reiten formula that the functor 
    $$ S := \Sigma \tau = \tau \Sigma $$
    is a Serre functor in $\DbRepQ$, i.e. 
    \begin{equation} \label{eq : isom def of Serre functor}
         \forall M,N \enspace \mathrm{Hom}(M,N) \simeq D \mathrm{Hom}(N,SM) 
      \end{equation}    
    where $D$ denotes the usual duality of $\mathbf{k}$-vector spaces. 
Recall the notion of an irreducible morphism as introduced by Auslander and Reiten in \cite{AR4}: a morphism $f : X \to Y$ in a given category is called \textit{irreducible} if $f$ is neither a split monomorphism nor a split epimorphism, and for any commutative diagram \[ \begin{tikzcd}
                                   & W \arrow[rd, "h"] &   \\
X \arrow[rr, "f"'] \arrow[ru, "g"] &                   & Y
\end{tikzcd} \] either $g$ is a split monomorphism or $h$ is a split epimorphism. (recall that a morphism $u$ is a split monomorphism when there is some $u'$ such that $u'u = 1$, and that $u$ is a split epimorphism when there is some $u'$ for which $uu' = 1$).

 For each $i \in I$, we will denote by $P_i$ (resp. $I_i$) the indecomposable projective (resp. injective) $\mathbf{k}Q$ module at the vertex $i$. Recall that their dimension vectors are given by 
$$ \boldsymbol{\dim} P_i = \sum_{ j \in I} p_Q(i,j) \alpha_j  \qquad 
  \boldsymbol{\dim} I_i = \sum_{ j \in I} p_Q(j,i) \alpha_j $$
  where $p_Q(k,l)$ denotes the number of oriented paths  in $Q$ from a vertex $k$ to a vertex $l$ of $Q$. 
Given an element $\beta \in \bigoplus_{i \in I} \mathbb{Z}_{\geq 0}\alpha_i$, we denote by $\mathrm{Supp}(\beta) := \{ i \in I \mid a_i>0 \}$ and by $\Qbeta$ the full subquiver of $Q$ having $\mathrm{Supp}(\beta)$ as set of vertices. We will then denote by $P_i^{\beta}$ (resp. $I_i^{\beta}$) denote the   the  indecomposable projective (resp. injective) $\mathbf{k}\Qbeta$ modules at the vertex $i$.

 \subsection{Homotopy categories}

 Given an additive category $\mathcal{H}$, we will denote by $\mathcal{K}^b(\mathcal{H})$ the bounded homotopy category of $\mathcal{H}$. Objects in $\mathcal{K}^b(\mathcal{H})$ will be denoted  
 $$ C_{\bullet} :  \enspace \cdots  \rightarrow C_{n-1} \xrightarrow{d_{n-1}} C_n \xrightarrow{d_n} C_{n+1} \rightarrow \cdots  $$
 and we will denote by $C_{\bullet}[1]$ the shift to the left of $C_{\bullet}$, i.e. the image of $C_{\bullet}$ under the suspension functor for the usual triangulated structure on $\mathcal{K}^b(\mathcal{H})$, i.e. the degree $n$ object of $C_{\bullet}[1]$ is given by $C_{n+1}$. We denote by $\chi(C_{\bullet})$ the Euler characteristics of a complex $C_{\bullet}$, defined as 
 $$ \chi(C_{\bullet}) := \sum_{n \in \mathbb{Z}}(-1)^n [C_n]  \enspace \in K_0(\mathcal{H}) $$
 where $K_0(\mathcal{H})$ stands for the split Grothendieck group of $\mathcal{H}$. Given two objects $C_{\bullet}$ and $D_{\bullet}$ and a morphism $u_{\bullet} : C_{\bullet} \rightarrow D_{\bullet}$ in $\mathcal{K}^b(\mathcal{H})$, we will denote by $\mathrm{Cone}(u_{\bullet})$ the mapping cone of $u_{\bullet}$, which is the object of $\mathcal{K}^b(\mathcal{H})$ whose degree $n$ object is $C_{n+1} \oplus D_{n}$ and whose differentials are given by 
 $$ d_n^{\mathrm{Cone}(u_{\bullet})} := 
  \begin{pmatrix}
      d_{n+1}^{C_{\bullet}} & u_{n+1} \\
       0 & -d_{n}^{D_{\bullet}}
  \end{pmatrix}
 $$
 Classically, we get a distinguished triangle in $\mathcal{K}^b(\mathcal{H})$:
 $$ C_{\bullet} \xrightarrow{u_{\bullet}} D_{\bullet} \rightarrow \mathrm{Cone}(u_{\bullet}) \rightarrow C_{\bullet}[1] . $$
 Assuming now that $\mathcal{H}$ comes with a symmetric monoidal structure $\otimes$, $\mathcal{K}^b(\mathcal{H})$ can then be naturally endowed with a structure of (symmetric) triangulated monoidal category. We do not provide  its explicit description here as it will not be needed. We simply mention that if $X$ is an object in $\mathcal{H}$ and $C_{\bullet} \in \mathcal{K}^b(\mathcal{H})$, we will denote by $X \otimes C_{\bullet}$ the chain complex whose degree $n$ object is $X \otimes C_n$ and whose $n$-th differential is $\mathrm{id}_X \otimes d_n^{C_{\bullet}}$. 

 \section{Representations of quantum affine algebras}
  \label{sec : reminders QAAs}

 This section provides necessary reminders about the representation theory of quantum affine algebras. We will mostly focus on the  symmetric affine case and will recall the definitions and main features of the Hernandez-Leclerc categories. 

 \subsection{Representations of quantum affine algebras and their $q$-characters} \label{sec : q-characters}

 Let $\mathfrak{g}$ be a simply-laced type simple complex Lie algebra and let $U_q(\widehat{\mathfrak{g}})$ denote the associated quantum affine algebra. The category $\mathcal{C}$ of finite-dimensional representations of $U_q(\widehat{\mathfrak{g}})$ is an abelian monoidal category which is not semisimple,  not braided, and that contains  infinitely (and even uncountably) many simple objects. More precisely, it was proved by Chari-Pressley \cite{CP95a} that the irreducible representations in $\mathcal{C}$ are classified (up to isomorphism) by a monoid $\mathcal{P}^+$ whose elements are called dominant monomials and which is defined as the free monoid generated by indeterminates $Y_{i,a} , i \in I , a \in \mathbb{C}^{*}$. Here $I$ denotes an indexing set of the simple roots of $\mathfrak{g}$. We will denote by $L(\mathfrak{m})$ the simple module corresponding to $\mathfrak{m}$ under this classification. The simple modules of the form $L(Y_{i,a})$ are called \textbf{fundamental representations}. Once a certain order has been fixed on $I \times \mathbb{C}^*$, the ordered tensor products  of fundamental representations are called \textbf{standard modules}. If $\mathfrak{m} = \prod_{i,a} Y_{i,a}^{u_{i,a}}$, we will denote by $\Delta(\mathfrak{m})$ the standard module given by the ordered tensor product of the $L(Y_{i,a})^{\otimes u_{i,a}}$. It is then known that for any dominant monomial $\mathfrak{m}$, the simple module $L(\mathfrak{m})$ is isomorphic to a quotient of $\Delta(\mathfrak{m})$. In this sense, the fundamental representations generate $\mathfrak{C}$ as a monoidal category.  Although $\mathcal{C}$ is not symmetric, its Grothendieck ring is commutative. Frenkel-Reshetikhin \cite{FR1} introduced an injective ring homomorphism 
 $$ \chi_q : \enspace K_0(\mathcal{C}) \longrightarrow \mathcal{Y} := \mathbb{Z}[Y_{i,a}^{\pm} , i \in I , a \in \mathbb{C}^*]  $$
 called the $q$-character morphism.  Essentially, the  Laurent monomials appearing in the $q$-character of a given object $M$ in $\mathcal{C}$ represent the eigenvalues of the action of the (loop) Cartan subalgebra of $U_q(\widehat{\mathfrak{g}})$ on $M$ and the coefficients of $\chi_q(M)$ encode the dimensions of the corresponding eigenspaces.  It is known that the  \textbf{renormalized $q$-character} of $M$ defined as $\mathfrak{m}^{-1} \chi_q(M)$ is a polynomial in the $A_{i,a}^{-1} , i \in I , a \in \mathbb{C}^
 *$ with constant term $1$, where  
\begin{equation}  \label{eq : def of Ais}
  \forall i \in I, a \in \mathbb{C^{*}}, \enspace     A_{i,a} := Y_{i,aq^{-1}}Y_{i,aq} \prod_{j \sim i}Y_{i,a}^{-1} . 
 \end{equation}

  \subsection{Hernandez-Leclerc categories} \label{sec : HL categories}

 Motivated by exhibiting deep connections between the representation theory of quantum affine algebras and the combinatorics of cluster algebras introduced by Fomin and Zelevinsky \cite{FZ1}, Hernandez-Leclerc defined several monoidal full subcategories of $\mathcal{C}$, which we now briefly recall. 
Let $\IZ$ denote the set of pairs
$$ \IZ := \{ (i,p) , i \in I , p \in i + 2 \mathbb{Z} \} $$
 and let $\YZ$ denote the ring
 $$ \YZ := \mathbb{Z}[Y_{i,p}^{\pm 1} , (i,p) \in \IZ]. $$
Following \cite{HL15}, we denote by $\CZ$ the smallest monoidal  full  subcategory of $\mathcal{C}$ containing the fundamental representations $L(Y_{i,q^p})$ for $p \in i + 2 \mathbb{Z}$. Following the common use, we will be writing $Y_{i,p}$ (resp. $A_{i,p}$) as a shorthand notation for $Y_{i,q^p}$ (resp. $A_{i,q^p}$) in all what follows. The simple objects of $\CZ$ are the modules $L(\mathfrak{m})$ where $\mathfrak{m}$ is a monomial in the $Y_{i,p}$. The category $\CZ$ admits various kinds of well-known short exact sequences that play an important role in the theory of monoidal categorifications of cluster algebras \cite{HL16,KKOP2}. The following special case will be useful to us later
\begin{equation} \label{eq : T-system}
   0 \longrightarrow \bigotimes_{j \sim i} L(Y_{j,p+1}) \longrightarrow L(Y_{i,p}) \otimes L(Y_{i,p+2}) \longrightarrow L(Y_{i,p}Y_{i,p+2}) \longrightarrow 0 .  
\end{equation}
Note that the middle term is the standard module $\Delta(Y_{i,p}Y_{i,p+2})$.
Frenkel-Reshetikhin's $q$-character morphism restricts to an injective ring homomorphism 
 $$ \chi_q : \CZ \longrightarrow \YZ . $$
 The polynomial identity obtained by applying $\chi_q$ to~\eqref{eq : T-system} is an instance of the so-called $T$-systems. 
 
  We now fix an orientation $Q$ of the Dynkin diagram of $\mathfrak{g}$ and $\xi$ a height function adapted to $Q$. Then we set, for every positive integer $l \geq 1$,
 $$  \Ixi^{(l)} := \{ (i,p) \in \IZ , \enspace \xi(i)-2l \leq p \leq \xi(i)\} \qquad \text{and} \qquad  \Ixi := \{ (i,p) \in \IZ , p \leq \xi(i) \} $$
 and we denote by $\Cxi^{(l)}$ (resp. $\CxiZ$) the smallest full monoidal subcategory of $\CZ$ containing all the fundamental representations $L(Y_{i,p})$ such that $(i,p) \in \Ixi^{(l)}$ (resp. $(i,p) \in \Ixi$). 
Finally, following \cite{HL10,HL16}, we define for each $M \in \CxiZ$ the \textbf{truncated $q$-character} $\tchi_q(M)$ of $M$  as the element of $\YZ$ obtained from $\chi_q(M)$ by dropping all the terms involving variables $Y_{i,p}$ with $p> \xi(i)$. This gives another injective ring homomorphism 
 $$ \tchi_q : \CxiZ \longrightarrow \YxiZ :=\mathbb{Z}[Y_{i,p}^{\pm 1} ,  (i,p) \in \Ixi] .  $$

\subsection{The case of the category $\Cxi^{(1)}$}

 In this paper, we will be particularly interested in the category $\Cxi^{(1)}$. The following truncated $q$-characters are known:
 $$ \tchi_q(L(Y_{i, \xi(i)})) = Y_{i, \xi(i)} \qquad \tchi_q(L(Y_{i, \xi(i)-2} Y_{i, \xi(i)})) = Y_{i, \xi(i)-2}Y_{i, \xi(i)} $$
 and furthermore, the renormalized truncated $q$-character of any object in $\Cxi^{(1)}$ is a polynomial (with constant term $1$) in the $A_{i, \xi(i)-1}^{-1} , i \in I$. For that reason we will often use the shorthand notation $A_i := A_{i, \xi(i)-1}$ for each $i \in I$. 
 When $\mathfrak{g}$ is not finite-dimensional, the fundamental representations of $U_q(\widehat{\mathfrak{g}})$ are not anymore finite-dimensional. Nonetheless, they can still be studied by considering the Hernandez-Jimbo category $\mathcal{O}$ \cite{HerJim}, and the category $\Cxi^{(1)}$ can then be defined as a subcategory of $\mathcal{O}$ in the same way as above.  

  \

\section{The hammock category $\HQ$}
\label{sec : hammock category}


  \subsection{Quasi-additive functions}

  Extending Gabriel's  terminology of additive functions from \cite{Gabriel}, we consider the following  class of integer-valued functions on $\ZQz$.

   \begin{deftn} \label{def : quasi-additive functions}
     A function $f : \ZQz \longrightarrow \mathbb{Z}$ will be called quasi-additive if $\widetilde{f}$ has finite support, i.e. $\widetilde{f}(x)=0$ on all but finitely many vertices of $\ZQ$.
   \end{deftn}
For each $x \in \ZQz$, we consider the unique  quasi-additive function $h_x$ determined by the requirements 
    \begin{enumerate}
        \item $h_x(y)= \dim \mathrm{Hom}_{\DbRepQ}(x,y)$  for every vertex $y$ on the section $Q_x$ of $\ZQ$,
        \item $ \widetilde{h_x} = \delta_x$,
    \end{enumerate}
   where $\delta_x$ denotes the function given by $\delta_x(y) := 1$ if $x=y$ and $0$ otherwise. For a given $x \in \ZQz$, the function $h_x$ should be viewed as a semi-infinite version of the \textbf{hammock functions} considered by Gabriel \cite{Gabriel} and Brenner \cite{Brenner}. With the notations above, a function is additive in the sense of \cite{Gabriel} if $\widetilde{f}=0$, hence the terminology "quasi-additive". 

\begin{example} \label{ex : quasi-additive function}
Consider the Dynkin quiver $Q$ of type $A_4$ given by $1 \rightarrow 2 \rightarrow 3 \leftarrow 4$. We picture below the function $h_x$ where $x$ is the vertex of $\ZQ$ outlined in red. 

 $$ \xymatrixrowsep{1pc} \xymatrixcolsep{1pc} \xymatrix{  
   {} &  \cdots   & 0 \ar[rd]   &    {}   &    1  \ar@{.}[ll] \ar[rd]   & {}  & 0 \ar@{.}[ll] \ar[rd] & {} & -1 \ar@{.}[ll]   \\
    \cdots    &  0 \ar[ru] \ar[rd]   &    {}      &  1 \ar@{.}[ll]  \ar[ru]  \ar[rd]  &{}  &  1 \ar@{.}[ll] \ar[ru]  \ar[rd] & {} & -1 \ar@{.}[ll]  \ar[ru] & \cdots \\
      0 \ar[ru] \ar[rd]  &     {}     &    \textcolor{red}{1} \ar[ru]  \ar@{.}[ll]   \ar[rd]   & {} & 1 \ar@{.}[ll]  \ar[rd] \ar[ru] & {} & 0 \ar[ru] \ar@{.}[ll] \ar[rd] & \cdots  \\
    \cdots   &   0  \ar[ru]   &  {}       &   1 \ar@{.}[ll]  \ar[ru]   & {} & 0  \ar@{.}[ll]   \ar[ru]& {} & 0 \ar@{.}[ll] & \cdots 
    } $$

\end{example}
   
 We now prove a couple of important  properties of these functions.

  \begin{lem} \label{lem : linear independece of hammock functions}
      The functions $h_x , x \in \ZQz$ are linearly independent.
  \end{lem}

 \begin{proof}
     Assume $r \geq 1$ and  $a_1, \ldots , a_r \in \mathbb{Z}$  and $x_1, \ldots , x_r \in \ZQz$ are such that $a_1 h_{x_1} + \cdots + a_r h_{x_r} = 0$. We can choose $i$ such that $\tau^{-1}x_i$ is to the left of the other $\tau^{-1}x_j$. Wlog $i=1$. Hence we have $h_{\tau^{-1}x_k}(x_1) = 0$ for all $k>1$ and $h_{x_1}(x_1) = 1$ by definition of the functions $h_x$. This implies $a_1=0$. The lemma thus follows by a straightforward induction. 
  \end{proof}

  \begin{lem} \label{lem : mutation for functions}
   For every $x \in \ZQz$ we have that 
   $$ h_x + h_{\tau^{-1}x} = \delta_x + \sum_{x \rightarrow y} h_y .  $$
  \end{lem}

   \begin{proof}
     Let us fix $x \in \ZQz$. It is then straightforward to check that 
     $$ \widetilde{\delta_x} = \delta_x + \delta_{\tau^{-1}x} - \sum_{x \rightarrow y} \delta_y  . $$
    It follows from this that the function 
    \begin{equation} \label{eq : aux function}
         h_x + h_{\tau^{-1}x} - \sum_{x \rightarrow y} h_y - \delta_x
     \end{equation}
     is additive and hence is entirely determined by its values on any section of $\ZQ$ (cf. \cite{Gabriel}). Let us consider the section starting at $x$. For any $z$ on that section, we have that $\dim_{\mathbf{k}} \mathrm{Hom}(x,z) =1$  so that $h_x(z) =1$ and $\mathrm{Ext}^1(z,x) = 0 $ which implies that $h_{\tau^{-1}x}(z) = 0$ using the Auslander-Reiten formula. On the other hand, $z$ belongs to exactly one of the sections starting at each $y$ such that there is an arrow $x \rightarrow y$ in $\ZQ$, unless $z$ is $x$ itself. In other words we have that $ \sum_{x \rightarrow y}h_y (z) = 1 $ if $z \neq x$, and $0$ if $z=x$. Putting all of this together we see that the function~\eqref{eq : aux function} vanishes on all the vertices of the section starting at $x$. The lemma is proved. 
   \end{proof}

 \subsection{The category $\MQ$}
  \label{sec : category MQ}
 
 We denote by $\IQ$ the set of all isomorphism classes of indecomposable objects in $\DbRepQ$. Recall that a multiset of elements of $\IQ$ is a collection $\{\tau^{-1}x_k , k \in K\}$ where $K$ is a set and $\tau^{-1}x_k$ is an indecomposable object in $\DQ$ for each $k \in K$ (equivalently, it can be viewed as a  function from $K$ to $\IQ$). Note that such a function shall usually not be injective; in other words, any given indecomposable object  in $\DbRepQ$ may appear several times (up to isomorphism) in the same collection (but only finitely many times !). In this work, we will exclusively be considering at most countable  multisets, i.e.  functions from a finite or  countable set $K$ to $\IQ$. 
  Given two multisets $X$ and $Y$ we can define their intersection $X \cap Y$ as the largest multiset contained in both $X$ and $Y$. 
 
 We define a $\mathbf{k}$-category $\tMQ$ as follows. The class of objects of $\tMQ$ consists of all pairs $(X,h)$ where $X$ is an at most countable multiset of elements of $\IQ$ and $h$ is a quasi-additive function on $\ZQz$. We then define the notion of Serre tilting of such pairs as follows. Recall that $S$ denotes the Serre functor $S = \tau \Sigma = \Sigma \tau$ in the category $\DbRepQ$. 
 
   \begin{deftn} \label{def : Serre tiltings}
 Given any indecomposable object $(X,h)$ in $\tMQ$ and any \textit{finite} subset $Y = \{ y_1 , \ldots  , y_r\} \subseteq X$, we define the \textbf{Serre tilting} of $(X,h)$ at $Y$ as the object of $\tMQ$ given by  
 $$ \mu_Y(X,h) := \left(  (X \setminus Y) \sqcup SY \enspace , \enspace  h - \sum_{y \in Y} \delta_y \right) $$
 where $SY$ stands for $SY := \{S y_1, \ldots , S y_r\}$.  For simplicity we will write $\mu_y(X)$ for $\mu_{\{y \}}(X)$.
   \end{deftn}
  
   We now define the class of morphisms of the category $\tMQ$. 
 Given two at most countable sets $K$ and $L$, we denote by $\mathrm{Bij}(K,L)$ the set of all bijections from $K$ to $L$. In particular $\mathrm{Bij}(K,L) = \emptyset$ if $K$ and $L$ are finite of different cardinalities. Moreover for any function $\sigma : K \rightarrow L$, we set 
  $$  \mathrm{Supp}_{X,Y}(\sigma) := \{k \in K \mid y_{\sigma(k)} \neq \tau^{-1}x_k \} .  $$
  Given two multisets $X = \{\tau^{-1}x_k, k \in  K\}$ and $Y = \{y_l , l \in L\}$ and two  quasi-additive functions $f$ and $g$ on $\ZQz$, we then define the morphism space from $(X,f)$ to $(Y,g)$ in $\widetilde{\MQ}$ as follows:
   \begin{equation} \label{eq : def of morphisms in CQ}
      \mathrm{Hom}_{\tMQ} \left( (X,f),(Y,g) \right)  := 
       \bigoplus_{  \substack{ \sigma \in \mathrm{Bij}(K,L)  \\  \sharp \mathrm{Supp}_{X,Y}(\sigma) < \infty }}  \bigotimes_{k \in \mathrm{Supp}_{X,Y}(\sigma)} \mathrm{Hom}_{\DQ}(\tau^{-1}x_k , y_{\sigma(k)})  
  \end{equation}
   Note that each summand on the right hand side is a tensor product of finitely many $\mathbf{k}$-vector spaces (as $\DbRepQ$ is a $\mathbf{k}$-category) so that $\tMQ$ is a $\mathbf{k}$-category.

    \begin{rk}
        In most of what follows we will be considering only the case where $X$ and $Y$ are finite with same cardinality so that
   $$ \mathrm{Hom}_{\tMQ} \left( (X,f),(Y,g) \right) := \bigoplus_{ \sigma \in \mathrm{Bij}(K,L)} \bigotimes_{k \in K} \mathrm{Hom}_{\DQ}(\tau^{-1}x_k , y_{\sigma(k)}) .  $$
    \end{rk}


 The category $\tMQ$ can be endowed with a symmetric monoidal structure as follows.   
   Given two  objects $(X,f)$ and $(Y,g)$ in $\tMQ$ with $X = \{\tau^{-1}x_k , k \in K\}$ and $Y = \{y_l , l \in L\}$, we set 
\begin{equation} \label{eq : monoidal structure on CQ}
 (X,f) \otimes (Y,g) :=  \left( X \sqcup Y , f + g \right) .  
\end{equation}   
   where $X \sqcup Y$ is the multiset defined by
   $$ X \sqcup Y := \{z_p , p \in K \sqcup L \} , \enspace z_p := \begin{cases}
       \tau^{-1}x_p & \text{if $p \in K$,} \\
       y_p & \text{if $p \in L$. }
   \end{cases} 
   $$
Moreover, given two non-trivial morphisms $\varphi_1 : (X_1,f_1) \rightarrow (Y_1,g_1) $ and $\varphi_2 : (X_2,f_2) \rightarrow (Y_2,g_2)$ in $\tMQ$, we  get a canonical non-trivial morphism $\varphi_1 \otimes \varphi_2 : (X_1 , f_1) \otimes (X_2,f_2) \rightarrow (Y_1,g_1) \otimes (Y_2,g_2)$ in $\tMQ$ as $(Y_1 \sqcup Y_2, g_1 + g_2) = \mu_{Z_1 \sqcup Z_2} (X_1 \sqcup X_2 , f_1 + f_2)$ where $Z_1$ and $Z_2$ are such that $(Y_i,g_i) = \mu_{Z_i}(\tau^{-1}x_i,f_i)$ for each $i \in \{1,2\}$. It is straightforward to check that this construction is natural and hence~\eqref{eq : monoidal structure on CQ} endows $\tMQ$ with a symmetric monoidal structure. 
 Finally, we denote by $\MQ$ the additive envelope of $\tMQ$. It follows from what precedes that $\MQ$ is a $\mathbf{k}$-linear symmetric monoidal category.

\begin{rk} \label{rk : indecomposability}
   We  believe the objects of $\tMQ$ can be shown to be  indecomposable in $\MQ$. We can prove it in the case where $Q$ is a Dynkin quiver but do not know how to extend this to arbitrary acyclic quivers.   
\end{rk}

\subsection{The category $\HZQ$} \label{sec : category CQ and KQ}

 We now introduce a certain full monoidal category $\HZQ$ of $\MQ$ that we call the hammock category. We begin by defining certain distinguished objects (that we will refer to as  hammock objects) in $\MQ$ that will be playing a crucial role in what follows.  For every  $x$ in $\IQ$, we define the $\textbf{hammock multiset}$ associated to $x$  as the multiset $H(x)$ of elements of $\IQ$ given by
  \begin{equation} \label{eq : def of hammock multisets}
       H(x) := \bigsqcup_{y \in \IQ} \bigsqcup_{m=1}^{\dim_{\mathbf{k}} \mathrm{Hom}(x,y)} \{y  \} .     
  \end{equation}
  
  In other words, $H(x)$ is the multiset containing $m_x(y)$ copies of each $y \in \IQ$, where $m_x(y) := \dim_{\mathbf{k}} \mathrm{Hom} (x,y)$ (in $\DbRepQ$) for each $y \in \IQ$.  Also, for every $x \in \IQ$, we define the \textbf{hammock object} associated to $x$ as the indecomposable object $Y(x)$ of $\MQ$ given by
    \begin{equation} \label{eq : def of hammock objects}
     Y(x) := \left(  H(x) , h_x \right) .   
    \end{equation}

\begin{example}
 Consider $Q$  a sink-source oriented quiver of type $D_4$, with the trivalent node labeled by $2$ and the monovalent nodes by $1,3,4$. Then the function $h_x$ associated to a vertex $x$ of color $2$ (outlined in red of the picture) looks like 
 $$
 \xymatrixrowsep{1pc} \xymatrixcolsep{1pc} 
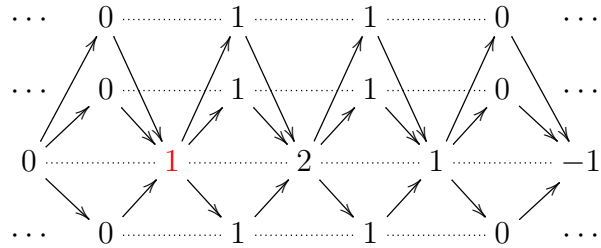
\xymatrix{  
   \cdots   & 0 \ar[rdd]     & {} &   1  \ar@{.}[ll] \ar[rdd]   & {}  & 1 \ar@{.}[ll] \ar[rdd] & {} & 0 \ar[rdd] \ar@{.}[ll]  &  \cdots  & {}   \\
    \cdots    &  0  \ar[rd]   &    {}      &  1 \ar@{.}[ll]   \ar[rd]  &{}  &  1 \ar@{.}[ll]   \ar[rd] & {} & 0 \ar[rd] \ar@{.}[ll]  &  \cdots  & {}  \\
      0 \ar[ru] \ar[ruu] \ar[rd]  &     {}     &    \textcolor{red}{1} \ar[ru]  \ar[ruu] \ar@{.}[ll]   \ar[rd]   & {} & 2 \ar@{.}[ll]  \ar[rd] \ar[ru] \ar[ruu] & {} & 1 \ar[ru]  \ar[ruu] \ar@{.}[ll] \ar[rd] & {} & -1 \ar@{.}[ll] \\
    \cdots   &   0  \ar[ru]   &  {}       &   1 \ar@{.}[ll]  \ar[ru]   & {} & 1 \ar@{.}[ll]   \ar[ru]& {} & 0 \ar@{.}[ll]\ar[ru] & \cdots  & {}
    } 
$$
 and the hammock multiset $H(x)$ is given by $$H(x) = \{x,y,z,t, \tau^{-1}x,\tau^{-1}x,\tau^{-1}y, \tau^{-1}z, \tau^{-1}t,\tau^{-2}x\}$$
 where $y,z,t$ denote the targets of the arrows of source $x$ in $\ZQ$. 
 We also refer to Figure~\ref{fig : hammocks} below for other examples of hammock multisets.
\end{example}

 \begin{example} \label{ex : morphisms in HQ}
     It will be convenient to represent the morphisms in $\MQ$ as braid diagrams. This way, composing two morphisms amounts to put the corresponding diagrams on top of each other, while the monoidal structure boils down to putting the two diagrams next to each other. For example, consider the quiver  $Q = 1 \rightarrow 2 \leftarrow 3$. We can depict $\DbRepQ$ as follows. 
     $$
 \xymatrixrowsep{1pc} \xymatrixcolsep{1pc} \xymatrix{  
    \cdots    &  P_1  \ar[rd]^g   &    {}      &  I_3 \ar@{.}[ll]   \ar[rd]  &{}  &  \Sigma P_3 \ar@{.}[ll]   \\
      P_2 \ar[ru]^f \ar[rd]_h  &     {}     &   I_2 \ar[ru]   \ar@{.}[ll]   \ar[rd]   & {} & \Sigma P_2 \ar@{.}[ll]   \ar[ru] \ar[rd]  & \cdots  \\
    \cdots   &   P_3  \ar[ru]_{i}   &  {}       &   I_1 \ar@{.}[ll]  \ar[ru]   & {} & \Sigma P_1 \ar@{.}[ll] 
    } 
$$
 Then consider for example the object $Y(P_2)$, whose multiset part is given by $H(P_2) = \{P_2,P_1,P_3,I_2\}$. We provide a couple of examples of morphisms in $\mathrm{Hom}_{\MQ}(Y(P_2), \mu_{P_2}Y(P_2))$ using this diagrammatic description. Note that $\mu_{P_2}H(P_2) = \{I_2,P_1,P_3,I_2\}$. 

  \begin{center}

 \begin{tikzpicture}[scale=0.8]
\node (a1) at (0,2) {$P_2$};
\node (b1) at (1.5,2) {$P_1$};
\node (c1) at (3,2) {$P_3$};
\node (d1) at (4.5,2) {$I_2$};

\node (a2) at (10,2) {$P_2$};
\node (b2) at (11.5,2) {$P_1$};
\node (c2) at (13,2) {$P_3$};
\node (d2) at (14.5,2) {$I_2$};

\node (x1) at (0,0) {$I_2$};
\node (y1) at (1.5,0) {$P_1$};
\node (z1) at (3,0) {$P_3$};
\node (t1) at (4.5,0) {$I_2$};

\node (x2) at (10,0) {$I_2$};
\node (y2) at (11.5,0) {$P_1$};
\node (z2) at (13,0) {$P_3$};
\node (t2) at (14.5,0) {$I_2$};

\draw[thick,->] (a1)--(y1) node[pos=0.3,above,font=\tiny] {$f$}    ;
\draw[thick,->] (b1)--(x1)  node[pos=0.7,below,font=\tiny] {$g$} ;
\draw[thick,double distance = 2pt] (c1)--(z1);
\draw[thick,double distance = 2pt] (d1)--(t1);

\draw[thick,->] (a2)--(x2) node[pos=0.3,left,font=\tiny] {$gf$} ;
\draw[thick,double distance = 2pt] (b2)--(y2) ;
\draw[thick,double distance = 2pt] (c2)--(z2);
\draw[thick,double distance = 2pt] (d2)--(t2);

 \end{tikzpicture}

  \end{center}
 
 \end{example}

   We will also need the following (indecomposable) objects in $\MQ$:
\begin{equation} \label{eq : def of F objects}
     \forall x \in \IQ, \enspace F(x) :=  \left( \{ Sx , \Sigma x \} , 0 \right) . 
\end{equation}

We then define the category $\HZQ$ as the smallest additive monoidal full subcategory of $\MQ$ whose class of objects contains $F(x)$ and $H(x)$ for each $x \in \IQ$ and is stable under Serre tiltings. 
The pairs of the form $(X_1,f_1) \otimes \cdots \otimes (X_r,f_r)$ where for each $k$, $(X_k,f_k)$ is of the form $Y(\tau^{-1}x_k)$  for some $x_k \in \ZQz$ will be called \textbf{dominant objects}.

 \begin{rk} \label{rk on dominant objects}
    It follows from Lemma~\ref{lem : linear independece of hammock functions} that a dominant object is entirely determined (up to isomorphism) by its quasi-additive function. Indeed, if $(X,h)$ is dominant then $\widetilde{h}$ takes only non negative values on $\ZQz$, and $X$ is given by 
    $$ X = \bigsqcup_{z \in \ZQz} H(z)^{\sqcup \widetilde{h}(z)} .  $$
     We will say that $h$ is dominant if $\widetilde{h}(z) \geq 0$ for all $z \in \ZQz$. 
 \end{rk}

Though its proof is relatively simple, the following statement is one of the crucial features of the construction proposed in this paper. From a strictly technical perspective, it is the main reason why performing Serre tiltings on hammock objects will eventually lead to strong connections with the theory of $q$-characters. We refer to Remark~\ref{rk : Serre tiltings vs Chari's action} below for further comments. 

 \begin{prop} \label{prop : mutation for objects in CQ}
     For each $x \in \IQ$ we have that 
     $$ \mu_xY(x) \otimes Y(\tau^{-1}x) \simeq  F(x) \otimes \bigotimes_{x \rightarrow z} Y(z) . $$
 \end{prop}

  \begin{proof}
  In view of Lemma~\ref{lem : mutation for functions},  the desired statement is equivalent to  the following identity of multisets 
  \begin{equation} \label{eq : identity of multisets}
  H(x) \sqcup H(\tau^{-1}x) = \{x , \Sigma x \} \sqcup \bigsqcup_{x \rightarrow z} H(z) . 
  \end{equation}
  This will follow from the fact that $x \xrightarrow{\iota} \bigoplus_{x \rightarrow z} z \xrightarrow{\pi} \tau^{-1}x \rightarrow \Sigma x$ is an Auslander-Reiten triangle in $\DbRepQ$. Indeed, let $y \in \IQ$ and assume $y$ is not isomorphic neither to $x$  nor $\Sigma x$. We are going to check that the sequence
  $$ 0 \longrightarrow \mathrm{Hom}(\tau^{-1}x,y) \longrightarrow \bigoplus_{x \rightarrow z} \mathrm{Hom}(z,y) \longrightarrow \mathrm{Hom}(x,y) \longrightarrow 0 $$
  is then a short exact sequence in the category of $\mathbf{k}$-vector spaces. 
  As $y \not\simeq x$, the fact that the above triangle is left almost split implies that any morphism from $x$ to $y$ factors through $\iota$, hence the surjectivity of the right most non trivial map in the above sequence. Let us now take $f$ a non trivial morphism in $\mathrm{Hom}(\tau^{-1}x,y)$ and assume $f \circ \pi=0$. By definition of Serre functors, we can consider a non trivial morphism $\overline{f} : S^{-1}y \rightarrow \tau^{-1}x$ such that $ f \circ \overline{f} \neq 0$.  Note that $\overline{f}$ is not an isomorphism, because in that case we would have $\tau^{-1}x \simeq S^{-1}y = \tau^{-1} \Sigma^{-1}y$ which contradicts our assumption that $y \not\simeq \Sigma x$. Therefore, as the triangle $x \rightarrow \bigoplus_{x \rightarrow z} z \rightarrow \tau^{-1}x \rightarrow \Sigma x$ is right almost split, $\overline{f}$ has to factor through $\pi$. But this implies $f \circ \overline{f}=0$ (as $f \circ \pi = 0$ by assumption) which is a contradiction. This shows the injectivity of the left most non trivial map in the desired short sequence. Finally the exactness of the middle part follows from the fact that $\mathrm{Hom}(-,y)$ is cohomological (equivalently, $\pi$ is a weak cokernel of $\iota$ in $\DbRepQ$). 
   As a consequence of the exactness of the above short exact sequence, we obtain 
   $$ \forall y \neq x, \Sigma x \enspace \dim \mathrm{Hom}_{\DbRepQ}(x,y) + \dim \mathrm{Hom}_{\DbRepQ}(\tau^{-1}x,y)  = \sum_{x \rightarrow z}  \dim \mathrm{Hom}_{\DbRepQ}(z,y)  . $$
   Therefore, in order to conclude that~\eqref{eq : identity of multisets} holds, it remains to observe that $x$ appears once in $H(x)$ (as $\dim \mathrm{Hom}(x,x)=1$) and does not appear in $H(\tau^{-1}x)$, while $\Sigma x$ appears once in $H(\tau^{-1}x)$ because
   $$  \dim \mathrm{Hom}_{\DbRepQ}(\tau^{-1}x,\Sigma x) =  \dim \mathrm{Hom}_{\DbRepQ}(x,Sx) = \dim \mathrm{Hom}_{\DbRepQ}(x,x) = 1 $$
   and does not appear in $H(x)$ as $ \dim \mathrm{Hom}_{\DbRepQ}(x,\Sigma x) = 0$. This finishes the proof of the proposition. 
  
  \end{proof}
  
 \begin{center}

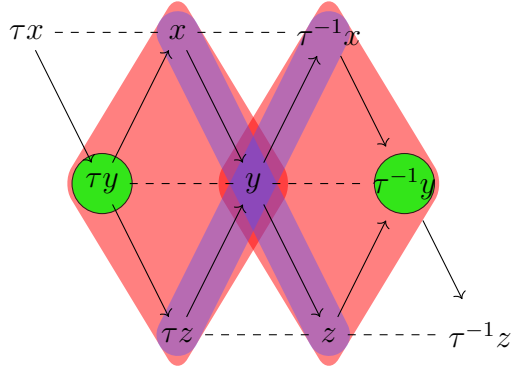
\begin{figure}
  
 \begin{tikzpicture}


style des sommets
\tikzset{
vertex/.style={
  circle,
  draw=none,
  inner sep=1.5pt
}
}

\tikzset{
hammock1/.style={fill=blue!50, opacity=.4, rounded corners=6pt, inner sep = 5pt}
}

\coordinate (x) at (0,2);
\coordinate (y) at (1,0);
\coordinate (z) at (2,-2);
\coordinate (taux) at (-2,2);
\coordinate (tauy) at (-1,0);
\coordinate (tauz) at (0,-2);
\coordinate (tauinvx) at (2,2);
\coordinate (tauinvy) at (3,0);
\coordinate (tauinvz) at (4,-2);

\node (sommetX) at (x) {$x$};
\node (sommetY) at (y) {$y$};
\node (sommetZ) at (z) {$z$};
\node (sommettauX) at (taux) {$\tau x$};
\node (sommettauY) at (tauy) {$\tau y$};
\node (sommettauZ) at (tauz) {$\tau z$};
\node (sommettauinvX) at (tauinvx) {$\tau^{-1} x$};
\node[vertex] (sommettauinvY) at (tauinvy) {$\tau^{-1} y$};
\node[vertex] (sommettauinvZ) at (tauinvz) {$\tau^{-1} z$};

 \draw[->] (sommetX)--(sommetY);
 \draw[->] (sommettauX)--(sommettauY);
 \draw[->] (sommettauinvX)--(sommettauinvY);
 \draw[->] (sommettauY)--(sommettauZ);
 \draw[->] (sommetY)--(sommetZ);
 \draw[->] (sommettauinvY)--(sommettauinvZ);
 \draw[->] (sommetY)--(sommettauinvX);
 \draw[->] (sommettauY)--(sommetX);
 \draw[->] (sommetZ)--(sommettauinvY);
 \draw[->] (sommettauZ)--(sommetY);

 \draw[-,dashed] (sommettauX)--(sommetX);  
 \draw[-,dashed] (sommetX)--(sommettauinvX);  
 \draw[-,dashed] (sommettauY)--(sommetY);     
 \draw[-,dashed] (sommetY)--(sommettauinvY);  
 \draw[-,dashed] (sommettauZ)--(sommetZ);  
 \draw[-,dashed] (sommetZ)--(sommettauinvZ);

 \node[fit=(tauy)(x)(y)(tauz), name=hammock2, inner sep=15pt,  draw=none] {};

 \node[fit=(y)(tauinvx)(tauinvy)(z), name=hammock2bis, inner sep=15pt,  draw=none] {};

 \begin{scope}[on background layer]
 \fill[red, rounded corners=7 pt, opacity=.5]
    (hammock2.north)--(hammock2.east)--(hammock2.south)--(hammock2.west)--cycle; 

  \fill[red, rounded corners=7pt, opacity=.5]  
   (hammock2bis.north)--(hammock2bis.east)--(hammock2bis.south)--(hammock2bis.west)--cycle; 

 \draw[blue!70, opacity=.4, line cap=round, line width=16pt] (x)--(z);
 
 \draw[blue!70, opacity=.4, line cap=round, line width=16pt] (tauz)--(tauinvx);

  \node[draw,fill=green, opacity=0.8, circle, inner sep = 8pt] at (tauy) {};

   \node[draw, fill=green, opacity=0.8, circle, inner sep = 8pt] at (tauinvy) {};
   
 \end{scope}

 \end{tikzpicture}

  \caption{Illustration of~\eqref{eq : identity of multisets} for a quiver of type $A_3$: the union of the hammock multisets of $\tau y$ and $y$ (in red) can be decomposed as the union of those of $x$ and $\tau z$ (in blue), together with the set containing only $\tau y$ and its shift $\tau^{-1}y$ (in green).}
\label{fig : hammocks} 

\end{figure}
 
\end{center}

 \begin{rk} \label{rk : Serre tiltings vs Chari's action}
 It follows from Proposition~\ref{prop : mutation for objects in CQ} that  if $Q$ admits a height function $\xi : I \rightarrow \mathbb{Z}$ and  $x \in \ZQz$ is labeled by $(i,p) , i \in I , p \in \mathbb{Z}$, then  denoting 
 $$ Y_{i,p} := [Y(x)] \qquad F_{i,p} := [F(x)] $$
 in the split Grothendieck group of $\HZQ$, we have 
 $$[ \mu_x Y(x)] = F_{i,p} Y_{i,p} A_{i,p+1}^{-1} . $$
  For that reason, Serre tiltings should be viewed as some categorical analogue of Chari's braid group action on $\YZ$ \cite{Chari02}. This is one of the main motivations for the introduction of Serre tiltings. Note that our setup involves the classes $[F(x)] , x \in \ZQz$ that do not appear in the theory of $q$-characters, which is the reason why the statement of Theorem~\ref{thm : thm2 intro} involves a specialization on the left hand side. 
\end{rk}

  \subsection{The subcategory $\HQ$}

   We now define a full monoidal subcategory $\HQ$ of $\HZQ$ by allowing only Serre tiltings on a well-chosen section of $\ZQ$. More precisely, let us fix an embedding of $\mathrm{mod}\,(\mathbf{k}Q)$ into $\DbRepQ$ and fix isomorphism classes $\{x_i, i \in I \}$ of the indecomposable projective $\mathbf{k}Q$ modules. As the irreducible morphisms between the $x_i$ form a copy of $Q^{op}$, it yields a section in $\ZQ$ that we also denote $\{x_i , i \in I\}$. We then define $\HQ$ as the smallest $^\mathbf{k}$-linear monoidal full subcategory of $\HZQ$ whose class of objects contains $Y(x_i), Y(\tau^{-1}x_i)$ as well as $F(x_i)$ for every $i \in I$, and that is stable under Serre tiltings $\mu_{x_j} , j \in I$. We will often use the shorthand notation $\mu_j := \mu_{x_j}$.  We denote by $K_0(\HQ)$ the split Grothendieck group of $\HQ$, and set the following notations:
\begin{equation} \label{eq : classes of objects in HQ}
 \forall i \in I, \quad  Y_{x_i} := [Y(x_i)] \quad  Y_{\tau^{-1}x_i}:= [Y(\tau^{-1}x_i)] \quad F_i := [F(x_i)] . 
\end{equation}

 \section{A result on complete exceptional sequences in $\mathrm{mod}\,\mathbf{k}Q$}
  \label{sec : exceptional sequences}

  This section is devoted to the proof of the following result, which we believe may be of independent interest from the perspective of the representation theory of quivers. 

\begin{thm} \label{thm : exceptional sequences}
 Let $Q = (Q_0,Q_1)$ be a finite acyclic quiver, and denote $n := \sharp Q_0$ and $m := \sharp Q_1$. There exists a family $\{E^{(0)} , \ldots , E^{(m)}\}$  of complete exceptional sequences in $\mathrm{mod}\,\mathbf{k}Q$ satisfying the following properties:
     \begin{enumerate}
         \item $E^{(0)}$ (resp. $E^{(m)}$) consists only of the indecomposable projectives (resp. injectives).   
         \item If we denote $E^{(t)} := (x_1^{(t)} , \ldots , x_n^{(t)})$ for each $0 \leq t \leq m$, then 
         $$ \mathrm{Hom}_{\mathrm{mod}\,\mathbf{k}Q} \left( x_i^{(t)} , x_i^{(t')} \right) \neq 0 \quad \text{for every $0 \leq t \leq t' \leq m$ and $i \in \{1, \ldots , n \}$.} $$
         \item There is a bijection $\alpha \in Q_1 \longmapsto t_{\alpha} \in \{0, \ldots, m-1 \}$ such that 
         $$ \alpha := i \rightarrow j \enspace \Rightarrow \enspace   x_j^{(t_{\alpha}+1)} \simeq x_i^{(t_{\alpha})} \quad \text{and} \quad \mathrm{Hom}_{\mathrm{mod}\,\mathbf{k}Q} \left( x_j^{(t_{\alpha})} , x_i^{(t_{\alpha}+1)}  \right) = 0 . $$
         \item For each $i \in I$ and every arrow $\alpha$ whose target (resp. source) is $i$, we have 
         $$ \forall 0 \leq t,t' \leq m, \qquad  \text{$t \leq t_{\alpha} < t'$ (resp. $t<t_{\alpha} \leq t'$)}  \enspace \Rightarrow  \enspace  x_i^{(t)} \not\simeq  x_i^{(t')}.  $$
     \end{enumerate}
\end{thm}

 Let us choose a complete exceptional sequence $E^{(0)}$ in $\mathrm{mod}\,\mathbf{k}Q$ consisting only of the indecomposable projectives, and let us consider $i \in Q_0$ such that $x_n^{(0)} \simeq P_i$. Note that in particular $i$ is a source in $Q$. Let $n > k_1 > \cdots >k_r \geq 1$ be such that $x_{k_1}^{(0)}, \ldots , x_{k_r}^{(0)}$ are respectively isomorphic to the indecomposable projectives at each of the neighbors of $i$ in $Q$. It will also be convenient to set $k_0 := n$.
 We introduce the following piece of notation:
 $$ \forall 0 \leq t \leq r, \enspace \beta_t := \boldsymbol{\dim} P_i - \sum_{1 \leq t' \leq t} \boldsymbol{\dim} x_{k_{t'}}^{(0)}  . $$
 Let us first establish a couple of straightforward technicalities. 
 \begin{lem} \label{lem : Euler forms}
    Let $ t \in \{1 , \ldots , r \}$ and $l \in \{1 , \ldots , n \}$. Then we have:
\begin{enumerate}
   \item If $l<k_t$, then  $\langle \beta_t ,  \boldsymbol{\dim} x_l^{(0)} \rangle = 0$ and $\langle \boldsymbol{\dim} x_l^{(0)}  , \beta_t   \rangle \geq  0$. Moreover, if $l=k_{t+1}$, then 
    $ \langle  \boldsymbol{\dim}x_{k_{t+1}}^{(0)} , \beta_t \rangle = 1 $.
    \item If  $l \geq k_t$, then $\langle\boldsymbol{\dim} x_l^{(0)}  , \beta_t   \rangle = 0 $ and  $\langle \beta_t ,  \boldsymbol{\dim} x_l^{(0)} \rangle  \leq 0$. Moreover, if $l=k_t$, then  $\langle \beta_t , \boldsymbol{\dim}x_{k_t}^{(0)} \rangle = -1$. 
\end{enumerate}
    
 \end{lem}

 \begin{proof}
      \textbf{Proof of (1).}
      If $l<k_t$ then $n > k_s \geq k_t >l$ for all $1 \leq s \leq t$ and hence $ \langle P_i,  x_l^{(0)}\rangle  = \langle x_{k_s}^{(0)}, x_l^{(0)} \rangle $ for all $1 \leq s \leq t$ as $E^{(0)}$ is an exceptional sequence.  Thus $\langle \beta_t , \boldsymbol{\dim} x_l^{(0)} \rangle = 0$. In addition to that, as $i$ is  a source in $Q$, we have 
      \begin{equation} \label{eq : identity of dim vectors}
      \beta_t = \alpha_i + \sum_{t'>t} \boldsymbol{\dim}x_{k_{t'}}^{(0)}    
      \end{equation}
      which implies $\langle \boldsymbol{\dim} x_l^{(0)} , \beta_t \rangle \geq 0$ because as $ x_l^{(0)}$ is projective, the Euler form $\langle  x_l^{(0)} , Y \rangle$ is non negative for any object $Y$ in $\mathrm{mod}\,\mathbf{k}Q$. Finally,   
      $$\langle \boldsymbol{\dim} x_{k_{t+1}}^{(0)} , \beta_t , \rangle =  \langle x_{k_{t+1}}^{(0)} , S_i \rangle + \sum_{t' \geq t+1} \langle x_{k_{t+1}}^{(0)} , x_{k_{t'}}^{(0)} \rangle =  1 $$
   because $\langle x_{k_{t+1}}^{(0)} , S_i \rangle=0$ while $\langle x_u^{(0)},x_v^{(0)} \rangle = 0$ for all $1 \leq v<u \leq r$ as $E^{(0)}$ is an exceptional sequence.  

\textbf{Proof of (2).}
If $l \geq k_t$ then $k_s < k_t \leq l$ for all $r \geq s>t$ and hence  $\langle \tau^{-1}\tau^{-1}x_l^{(0)} , x_{k_s}^{(0)} \rangle = 0 $ for all such $s$ as $E^{(0)}$ is an exceptional sequence. As $\langle \tau^{-1}\tau^{-1}x_l^{(0)} , S_i \rangle = 0$, we obtain $\langle \boldsymbol{\dim}\tau^{-1}\tau^{-1}x_l^{(0)} , \beta_t \rangle = 0$ using~\eqref{eq : identity of dim vectors}. On the other hand, we have $\langle P_i , \tau^{-1}\tau^{-1}x_l^{(0)} \rangle  = 0$ and hence from the definition of $\beta_t$ it immediately follows that $\langle \beta_t , \boldsymbol{\dim} \tau^{-1}\tau^{-1}x_l^{(0)} \rangle \leq 0$ as the Euler form $\langle \tau^{-1}\tau^{-1}x_l^{(0)} , Y \rangle$ is non negative for any object $Y$ in $\mathrm{mod}\,\mathbf{k}Q$. Finally, 
$$  \langle \beta_t , \boldsymbol{\dim}x_{k_t}^{(0)} \rangle = \langle \beta_{t-1} - \boldsymbol{\dim}x_{k_t}^{(0)} , \boldsymbol{\dim}x_{k_t}^{(0)} \rangle = \langle \beta_{t-1} , \boldsymbol{\dim}x_{k_t}^{(0)} \rangle -1 = -1 $$
applying the previous bullet point with $t-1$ instead of $t$, and $l=k_t$. 
 \end{proof}

 \begin{lem} \label{lem : positivity of Euler form}
  Let  $1 \leq t \leq t' \leq r$, let $l \in \{k_t , \ldots , n \}$ and let $Y$ be an indecomposable object in $\mathrm{mod}\,\mathbf{k}Q$ such that $\mathrm{Hom}(x_l^{(0)},Y) \neq 0$ and $\langle \boldsymbol{\dim}Y , \beta_{t'} \rangle = 0$.  Then  the Euler form $\langle \boldsymbol{\dim}x_l^{(0)}  + \lambda \beta_t , s_{\beta_{t'}} \boldsymbol{\dim} Y \rangle $ is positive for any $\lambda \geq 0$. 
 \end{lem}


  \begin{proof}
   Denote $\mu := - \langle \beta_{t'} , \boldsymbol{\dim}Y \rangle$. Then
   $$\langle \boldsymbol{\dim}x_l^{(0)} , s_{\beta_{t'}} \boldsymbol{\dim} Y \rangle = \langle x_l^{(0)} , Y \rangle + \mu  \langle \boldsymbol{\dim} x_l^{(0)} , \beta_{t'} \rangle  > 0 $$
   as $\mathrm{Hom} \left( x_l^{(0)} , Y  \right)  \neq 0$ by assumption and $\langle \boldsymbol{\dim} x_l^{(0)} , \beta_{t'} \rangle = 0 $ by Lemma~\ref{lem : Euler forms} (2) given that $l \geq k_t \geq k_{t'}$. On the other hand, using Lemma~\ref{lem : Euler forms}, a straightforward induction shows that $\langle \beta_t , \beta_t \rangle = 1$ and hence (again using Lemma~\ref{lem : Euler forms}) $\langle \beta_t , \beta_s \rangle = 1$  for all $1 \leq t \leq s \leq r$.  Thus we get
    $$ \langle \beta_t , \boldsymbol{\dim} Y + \mu \beta_{t'} \rangle = \langle \beta_t , \boldsymbol{\dim} Y \rangle + \mu =  \langle \beta_t - \beta_{t'} , \boldsymbol{\dim} Y \rangle = \sum_{t<s \leq t'} \langle x_{k_s}^{(0)} , Y \rangle \geq 0 . $$
    All together, we conclude that  the desired Euler form is positive. 
    
   \end{proof}

  \begin{prop} \label{prop : sequences for arrows incident to i}
  For every $0 \leq t \leq r$, there is a complete exceptional sequence $E^{(t)} = (x_1^{(t)} , \ldots , x_n^{(t)})$ such that:
  $$ \forall 1 \leq l \leq n, \enspace  \boldsymbol{\dim} x_l^{(t)} = 
  \begin{cases}
     \boldsymbol{\dim}x_l^{(0)} & \text{if $1 \leq l < k_t$,} \\
    s_{\beta_t} \boldsymbol{\dim} x_l^{(0)} & \text{if $k_t  \leq l < n$,} \\
     \beta_t & \text{if $l=n$.}
  \end{cases}$$
  \end{prop}

   \begin{proof}
   The proof goes by induction on $t$. For $t=0$ there us nothing to prove (recall that $k_0 := n$).  Let us now assume it holds at rank $t$. We shall perform $ \sigma_{n-1}^{-1} \cdots \sigma_{k_{t+1}+1}^{-1} \sigma_{k_{t+1}} \cdots \sigma_{n-1}$ on $E^{(t)}$. First, we perform $\sigma_{k_{t+1}+1} \cdots \sigma_{n-1}$ on $E^{(t)}$: we obtain a complete exceptional sequence $E' = (X'_1 , \ldots , X'_n)$ whose objects are such that 
   $$ \boldsymbol{\dim} X'_j  \simeq \begin{cases}
       \boldsymbol{\dim} x_j^{(0)} & \text{if $j \leq k_{t+1}$,}\\
       \beta_t & \text{if $j = k_{t+1}+1$,} \\
       X_{j-1}^{(0)} & \text{if $j>k_{t+1}+1$. }
   \end{cases} 
   $$
   Then, performing the action of the elementary generator $\sigma_{k_{t+1}}$ on $E'$ yields a complete exceptional sequence $E''$ which is the same as $E'$ except that $ \boldsymbol{\dim}X_{k_{t+1}}'' = \beta_t$ while 
   $$\boldsymbol{\dim} X_{k_{t+1}+1}''  = \beta_t - \langle \boldsymbol{\dim} x_{k_{t+1}}^{(0)} , \beta_t , \rangle \boldsymbol{\dim}x_{k_{t+1}}^{(0)} =  \beta_t - \boldsymbol{\dim}x_{k_{t+1}}^{(0)} =   \beta_{t+1}  $$
   using Lemma~\ref{lem : Euler forms} (1). 
    Finally, we perform the action of $\sigma_{n-1}^{-1} \cdots \sigma_{k_{t+1}+1}^{-1}$ on $E''$ and we define $E^{(t+1)}$ as the resulting (complete) exceptional sequence. We obtain   $$  \boldsymbol{\dim} x_j^{(t+1)}  = \begin{cases}
       x_j^{(0)} & \text{it $j< k_{t+1}$,} \\
       \beta_t & \text{if $j = k_{t+1}$}, \\
       s_{\beta_{t+1}} \boldsymbol{\dim} x_j^{(0)} & \text{if $k_{t+1} < j < n$,} \\
       \beta_{t+1} & \text{if $j=n$.}
   \end{cases} $$
  Thus to conclude, it suffices to note that by definition we have   $\beta_t = \beta_{t+1} + \boldsymbol{\dim}x_{k_{t+1}}^{(0)}$, which  coincides with $s_{\beta_{t+1}} \boldsymbol{\dim}x_{k_{t+1}}^{(0)}$ by Lemma~\ref{lem : Euler forms} (1). 
   \end{proof}

 \begin{cor} \label{cor : properties satisfied for arrows incident to i}
  For every $1 \leq j \leq n$, we have $\mathrm{Hom} \left( x_j^{(t)} , x_j^{(t')} \right) \neq 0$ for all  $1 \leq t \leq t' \leq r$. Moreover, for every $1 \leq t \leq r$, we have $x_{k_t}^{(t)} \simeq x_n^{(t-1)}$   and $\mathrm{Hom} \left( x_{k_t}^{(t-1)} , x_n^{(t)} \right) = 0 $. 
 \end{cor}

  \begin{proof}
   We have already noted (cf. the proof of Lemma~\ref{lem : positivity of Euler form}) that $ \langle \beta_t , \beta_{t'} \rangle = 1 $  if $t \leq t'$, so the first statement holds for $j=n$. Let us now assume $j<n$. Combining Proposition~\ref{prop : sequences for arrows incident to i} together with Lemma~\ref{lem : Euler forms} (2), we see that $\boldsymbol{\dim}x_j^{(t)}$ can be written as $\boldsymbol{\dim} x_j^{(0)} + \lambda_{j,t} \beta_t$ with $\lambda_{j,t} \geq 0$. If $j < k_{t'}$ then $j< k_t$ as well and hence $\mathrm{Hom} \left( x_j^{(t)} , x_j^{(t')} \right) = \mathrm{Hom} \left(  x_j^{(0)} , x_j^{(0)} \right) \neq 0 $. If $l \geq k_{t'}$, then by Lemma~\ref{lem : Euler forms}(2) the equality $\langle \boldsymbol{\dim} x_j^{(0)} , \beta_{t'} \rangle = 0$ holds and hence applying Lemma~\ref{lem : positivity of Euler form}  with $l=j,  Y = x_j^{(0)}$ and $\lambda = \lambda_{j,t}$ we obtain
   $$ \langle x_j^{(t)} , x_j^{(t')} \rangle = \langle  \boldsymbol{\dim} x_j^{(0)} + \lambda_{j,t} \beta_t , s_{\beta_{t'}} x_j^{(0)} \rangle > 0 $$
   which establishes the first desired statement. 
For the second statement, by Proposition~\ref{prop : sequences for arrows incident to i} together with Lemma~\ref{lem : Euler forms} (1), we have 
$$ \boldsymbol{\dim}x_{k_t}^{(t)} = s_{\beta_t} \boldsymbol{\dim}x_{k_t}^{(0)} = \boldsymbol{\dim}x_{k_t}^{(0)} + \beta_t = 
\beta_{t-1} =  \boldsymbol{\dim} X_n^{(t-1)}  . $$
We also have  $ \langle \boldsymbol{\dim} x_{k_t}^{(0)} , \beta_t \rangle = 0$ by Lemma~\ref{lem : Euler forms} (1),  which implies $ \mathrm{Hom} \left(  x_{k_t}^{(t-1)} , x_n^{(t)} \right)= 0 $.

  \end{proof}

 \begin{proof}[\textbf{Proof of Theorem~\ref{thm : exceptional sequences}.}]
  We obviously have $\beta_r = \alpha_i  = \boldsymbol{\dim}S_i = \boldsymbol{\dim}I_i$ as $i$ is a source in $Q$. Therefore Proposition~\ref{prop : sequences for arrows incident to i} implies that the exceptional sequence $E^{(r)}$ takes the form $(\Sigma_i x_1^{(0)} , \ldots , \Sigma_i x_{n-1}^{(0)} , I_i)$ where $\Sigma_i$ denotes the BGP reflection functor at $i$. Note also that $x_1^{(0)} , \ldots , x_{n-1}^{(0)}$ are the indecomposable projectives other than $P_i$, so denoting by $Q'$ the full subquiver of $Q$ obtained by deleting the node $i$ as well as all the arrows incident to $i$, one can view $(x_1^{(0)} , \ldots , x_{n-1}^{(0)})$ as a complete exceptional sequence in $\mathrm{mod}\,\mathbf{k}Q'$ consisting of the indecomposable projectives. By induction on the number of vertices, we have a family $E'^{(r)} , \ldots , E'^{(m)}$ of complete exceptional sequences in $\mathrm{mod}\,\mathbf{k}Q'$ satisfying all the desired properties.  Then for each $r<t \leq m$, we define $E^{(t)}$ as the sequence whose $n-1$ first objects are the images of those of $E'^{(t)}$ under the reflection functor $\Sigma_i$, and whose $n$th object is the injective $I_i$. In the case $t=r$ this procedure gives $E^{(r)}$ by what has just been said before. We can already note that $E^{(r+1)} , \ldots , E^{(m)}$ are (complete) exceptional sequences, because the proof of Proposition~\ref{prop : sequences for arrows incident to i} shows that each of them can be obtained from $E^{(r)}$ by performing a sequences of braids involving only $\sigma_1^{ \pm} , \ldots , \sigma_{n-2}^{ \pm}$ (so $I_i$ will always remain the last object of all of these sequences). 
  
  Let us now verify that putting together the sequences $E^{(0)} , \ldots , E^{(m)} $, Properties (1)-(5) hold. Property (1) holds because by the induction assumption, $E'^{(r)}$ consists of all the indecomposable injectives $I_j , j \neq i$; as  $\Sigma_i I_j \simeq I_j$ in that case, it follows that $E^{(m)}$ consists of all the indecomposable injectives of $\mathrm{mod}\,\mathbf{k}Q$. Property (3) immediately follows from the induction assumption together with Corollary~\ref{cor : properties satisfied for arrows incident to i}. Property (2) follows from the induction assumption if $r<t \leq t' \leq m$ and from Corollary~\ref{cor : properties satisfied for arrows incident to i} if $0 \leq t \leq t' \leq r$ so that it remains only to check it if $t \leq r < t'$. Keeping the notations of the previous Corollary, we need to show that $\langle  \boldsymbol{\dim }\tau^{-1}x_j^{(0)} + \lambda_{j,t} \beta_t , s_i  \boldsymbol{ \dim } Y \rangle >0$ where $Y$ is an indecomposable object in $\mathrm{mod}\,\mathbf{k}Q$ supported on $Q'$ and such that $\boldsymbol{\dim} Y \geq \boldsymbol{\dim } \tau^{-1}x_j^{(0)}$. In particular $\langle Y,S_i \rangle = 0$ and $\langle \tau^{-1}x_j^{(0)}  , Y \rangle >0 $. Using~\eqref{eq : identity of dim vectors} and the fact that $\langle P,X \rangle$ is non negative for any object $X$ if $P$ is projective, we get 
  \begin{align*}
      \langle x_j^{(0)} + \lambda_{j,t} \beta_t , s_i \boldsymbol{\dim} Y \rangle & > \lambda_{j,t}  \langle \beta_t , s_i \boldsymbol{\dim} Y \rangle  \geq \lambda_{j,t} \langle \alpha_i , s_i \boldsymbol{\dim} Y \rangle \\
      &= \lambda_{j,t} \langle \alpha_i, \boldsymbol{\dim} Y - \langle S_i, Y \rangle \alpha_i \rangle  = 0 . 
  \end{align*}
Finally Let us check Property (4). Keeping the same notations as in Proposition~\ref{prop : sequences for arrows incident to i}, we first need to check that $x_i^{(t)} \not\simeq x_i^{(t')} $ for all $0 \leq t < t' \leq r$. This is equivalent to showing that $\beta_t \neq \beta_{t'}$ if $0 \leq t < t' \leq r$. We  shall show that $\langle \beta_{t'} , \beta_t \rangle  \leq 0$, which will in particular imply that $\beta_t \neq \beta_{t'}$. Recall that as $t<t'$ we have $k_t > k_{t'}$. Thus,
\begin{align*}
    \langle \beta_{t'} , \beta_t \rangle   &= \langle \beta_{t'} , \alpha_i + \sum_{s>t} \boldsymbol{\dim}x_{k_s}^{(0)} \rangle = 1 + \sum_{s>t} \langle \beta_{t'} , \boldsymbol{\dim}x_{k_s}^{(0)} \rangle \\
    &= 1 + \langle \beta_{t'} , \boldsymbol{\dim} x_{k_{t'}}^{(0)} \rangle + \sum_{t<s<t'} \langle \beta_{t'} , \boldsymbol{\dim}x_{k_s}^{(0)} \rangle \leq 0
\end{align*}
because $ \langle \beta_{t'} , \boldsymbol{\dim}x_{k_s}^{(0)} \rangle = 0 $ if $s \geq t'$ by Lemma~\ref{lem : Euler forms}  (1), and $ \langle \beta_{t'} , \boldsymbol{\dim} x_{k_{t'}}^{(0)} \rangle = -1 $ and $ \langle \beta_{t'} , \boldsymbol{\dim}x_{k_s}^{(0)} \rangle \leq 0 $ for each $t<s<t'$, by Lemma~\ref{lem : Euler forms} (2). Thus Property (4) holds at the vertex $i$. Let us now fix $1 \leq s \leq r$ and  show that $x_{k_s}^{(t)} \not \simeq x_{k_s}^{(t')}$ is $t \leq s-1 < t'$ (recall that by construction, $t_{i \rightarrow k_s} = s-1$ for every $1 \leq s  \leq r$). First it is worth noting that in view of the property (1), it is enough to assume $r' \leq r$.  We then have $k_t < k_s \geq k_{t'}$ and thus by Proposition~\ref{prop : sequences for arrows incident to i} :
$$ \boldsymbol{\dim} x_{k_s}^{(t)}  = \boldsymbol{\dim} x_{k_s}^{(0)} \qquad \text{and} \qquad  \boldsymbol{\dim} x_{k_s}^{(t')}  = s_{\beta_{t'}} \boldsymbol{\dim} x_{k_s}^{(0)}  . $$
 So we simply need to check that the Cartan pairing of $\beta'$ with $\boldsymbol{\dim}x_{k_s}^{(0)}$ is non zero. As $k_s \geq k_{t'}$,  Lemma~\ref{lem : Euler forms} (2) implies that $\langle x_{k_s}^{(0)} , \beta_{t'} \rangle =0$, while the other Euler form is easily seen to be negative:
 \begin{align*}
     \langle \beta_{t'} , \boldsymbol{\dim}x_{k_s}^{(0)} \rangle &=  \langle P_i , x_{k_s}^{(0)} \rangle - \sum_{1 \leq u \leq t'} \langle x_{k_u}^{(0)} , 
  x_{k_s}^{(0)} \rangle  \leq \langle x_{k_s}^{(0)} , x_{k_s}^{(0)} \rangle = -1 \end{align*}
  because $s \leq t'$. So we have shown that Property (4) holds for all arrows incident to $i$. The remaining cases depend only on the subquiver $Q'$ and hence follow by induction. 
  This concludes the proof of Theorem~\ref{thm : exceptional sequences}. 
\end{proof}

 \begin{rk} \label{rk : ineq on t}
     With the same notations as above, we have $t_{\alpha} < t_{\beta}$ for any arrows $\alpha, \beta$ in $Q_1$ such that the target of $\alpha$ coincides with the source of $\beta$. Indeed, keeping the same notations as in the proof of Theorem~\ref{thm : exceptional sequences}, we have (by construction) $t_{i \rightarrow k_t} < t_{k_t \rightarrow l}$ for any $1 \leq t \leq r$ and any $l$ (note that $l \neq i$ as $Q$ is acyclic). This shows the desired statement in the case the source of $\alpha$ is $i$. If the source of $\alpha$ is any vertex other than $i$, then the statement is actually dealing with  the subquiver $Q' $ (where the vertex $i$ as well as all arrows  incident to it have been deleted), so that  the corollary follows by induction. 
 \end{rk}

\begin{example}  \label{ex : exceptional sequences}
Consider the (tame) quiver
\(
Q=
\vcenter{\hbox{
\xymatrix@C=0.8em@R=0.6em{
& 2 \ar[rd] & \\
1 \ar[ru] \ar[rd] & & 4 \\
& 3 \ar[ru] &
}
}}
\).
 It is an easy exercise to check that the columns of the following table provide exceptional sequences meet all the requirements  from Theorem~\ref{thm : exceptional sequences}. We write dashed arrows to signify the existence of non zero morphisms in $\mathrm{mod}\,\mathbf{k}Q$ (and simply dashed edges when they are isomorphisms).
 $$ \xymatrixrowsep{6pt}
 \xymatrix{P_4 \ar@{..}[r] &  P_4 \ar@{.}[r]  & P_4 \ar@{.>}[r] & x \ar@{.>}[r] & I_4 \\
 P_3 \ar@{..}[r] &  P_3  \ar@{.>}[r] & y \ar@{.>}[r] & I_4 \ar@{.>}[r] & I_3 \\
 P_2 \ar@{.>}[r] &  P_1  \ar@{.>}[r] & x \ar@{.>}[r] & I_2 \ar@{..}[r] & I_2 \\
 P_1 \ar@{.>}[r] &  y  \ar@{.>}[r] & I_1 \ar@{..}[r] & I_1 \ar@{..}[r] & I_1
 }
 $$
 Here $x$ and $y$  denote indecomposable $\mathbf{k}Q$ modules of respective dimension vectors $\boldsymbol{\dim}x = \alpha_1 + \alpha_2 + \alpha_4$ and $\boldsymbol{\dim}y = \alpha_1 + \alpha_3 + \alpha_4$. Note also that $t_{1 \rightarrow 2} = 0, t_{1 \rightarrow 3} = 1, t_{2 \rightarrow 4} = 2, t_{3 \rightarrow 4} = 3$. 

\end{example}

 \section{The category $\RQ$}
 \label{sec : the category RQ}

  In this section, we construct the category $\RQ$ that will serve as underlying category for the construction of exact complexes in the next section. The definition and properties of the objects and morphisms in $\RQ$ crucially rely on the family of exceptional sequences provided by Theorem~\ref{thm : exceptional sequences}. We also define an additive functor $\mathcal{D}_Q$ allowing to go  from $\RQ$ back to the hammock category defined in Section~\ref{sec : hammock category}. 

   \subsection{Inclusions of multisets}

    For each $i \in I$ and $0 \leq t <m$, we set
    $$ \pi_i := \{x_i^{(s)} , 0 \leq s < m\} \quad \pi_i^{\leq t} := \{x_i^{(s)} , 0 \leq s \leq t \} \quad \pi_i^{<t} := \{x_i^{(s)} , 0 \leq s <t \}  $$
    viewed as sets not multisets, i.e. the elements of $\pi_i$ (resp. $\pi_i^{\leq t} , \pi_i^{<t}$) are \textit{pairwise distinct}. Here by convention $\pi_i^{<0} := \emptyset$. Denoting $m_i := \sharp \pi_i$ we write $ \pi_i := \{x_{i,1} , \ldots , x_{i,m_i}\}$. In particular by construction we have $x_{i,1} \simeq x_i^{(0)} \simeq  x_i$ and $x_{i,m_i} \simeq x_i^{(m-1)} \simeq Sx_i$.   Recall the bijection $\alpha \in Q_1 \mapsto t_{\alpha} \in \{0, \ldots , m-1 \}$ provided by Theorem~\ref{thm : exceptional sequences} (3). For the sake of legibility of certain notations used below, it will be convenient to consider the quiver $\overline{Q}$ whose set of vertices is $I \sqcup \{ \ast \}$ and whose set of arrows consist of that of $Q$ together with one arrow $i \rightarrow \ast$ for every $i$ such that $i$ is a sink in $Q$. We then define an integer $t_{\alpha} \in \{0, \ldots , m \}$ for each arrow $\alpha$ in $\overline{Q}$ by setting 
    $$ \forall \alpha \in \overline{Q}_1 , \enspace t_{\alpha} := \begin{cases}
        t_{\alpha} & \text{if $\alpha \in Q_1$,} \\
        m & \text{otherwise.}
    \end{cases} $$
    We then extend this notation to any \textit{non  lazy} path $p$ in $\overline{Q}$ by setting $ t_p := t_{i \rightarrow j}$ if $p := i \rightarrow j \rightarrow \cdots \rightarrow k$ with $k \in I \sqcup \{ \ast \}$. In what follows  $\mathrm{in}_Q(i)$ (resp. $\mathrm{out}_Q(i)$) will stand for  the set of all oriented paths in $\overline{Q}$ arriving at $i$ (resp. starting at $i$). We then introduce the following multisets:
    \begin{equation} \label{eq : def of barpi}
       \forall \alpha \in \overline{Q}_1, \enspace \barpi_{\alpha} := \pi_{s(\alpha)}^{\leq t_{\alpha}} \sqcup \bigsqcup_{  \substack{ p \in \mathrm{in}(s(\alpha)) \\ \text{$p$ non lazy} }} \pi_{s(p)}^{< t_p} . 
    \end{equation}

 \begin{example}
 In the situation of Example~\ref{ex : exceptional sequences}, we have for example $\pi_4 = \{P_4,x,I_4\}, \pi_3 = \{P_3,y,I_4,I_3\}, \pi_2=\{P_2,P_1,x,I_2\}, \pi_1 = \{P_1,y,I_1\}$.  Then taking for instance $\alpha = 4 \rightarrow *$:
  \begin{align*}
  \barpi_{\alpha} = \pi_4 \sqcup \pi_2^{<t_{2 \rightarrow 4}} \sqcup \pi_3^{<t_{3 \rightarrow 4}} \sqcup \pi_1^{<t_{1 \rightarrow 3}} \sqcup \pi_1^{< t_{1 \rightarrow 2}}  
  = \{P_4,x,I_4,P_2,P_1,P_3,y,P_1\} . 
  \end{align*}
  Note that $P_1$ appears twice in $\barpi_{\alpha}$, which is due to the fact there are two oriented paths from $1$ to $4$ in $Q$ (cf. Corollary~\ref{cor : multisets of projectives} below). 
 \end{example}

The next proposition will be crucial in what follows. 

 \begin{prop} \label{prop : multisets barpi}
    For every arrow $\alpha \in \overline{Q}_1 $, the following inclusions of multisets hold:
    $$  \{x_{s(\alpha)}\}  \sqcup \bigsqcup_{\beta \in \overline{Q}_1 \cap \mathrm{in}(s(\alpha))} \barpi_{\beta} \enspace     \subseteq   \enspace    \barpi_{\alpha} \enspace   \subseteq  \enspace   \pi_{s(\alpha)} \sqcup \bigsqcup_{  \substack{ p \in \mathrm{in}(s(\alpha)) \\ \text{$p$ non lazy} }} \pi_{s(p)}^{< t_p}  \enspace   \subseteq \enspace   H(x_{s(\alpha)}) . $$   
 \end{prop}

  \begin{proof}
  The middle inclusion is trivial.  Let us then deal with the leftmost inclusion. Let us fix an arrow $\alpha \in Q_1$ and denote $i := s(\alpha)$.  We can decompose the set of all non lazy paths arriving at $i$ as those of length one i.e. the arrows of target $i$ on the one hand, and the non lazy paths arriving at one of the sources of these arrows on the other hand. Consequently, the desired inclusion boils down to 
      $$ \{x_i\} \sqcup \bigsqcup_{\beta \in \overline{Q}_1 \cap \mathrm{in}(i)} \pi_{s(\beta)}^{\leq t_{\beta}}  \subseteq \pi_i^{\leq t_{\alpha}} \sqcup \bigsqcup_{\beta \in \overline{Q}_1 \cap \mathrm{in}(i)} \pi_{s(\beta)}^{< t_{\beta}} .   $$
      Now for any arrow $\beta \in \overline{Q}_1 \mathrm{in}(i)$, denoting $j := s(\beta)$ we have $t_{\beta}< t_{\alpha}$ by Remark~\ref{rk : ineq on t} so that $x_i^{(t_{\beta}+1)} \in \pi_i^{\leq t_{\alpha}}$ and moreover $ x_i^{(t_{\beta}+1)} \simeq x_j^{(t_{\beta})}$
      by Theorem~\ref{thm : exceptional sequences} (3). Note also that Theorem~\ref{thm : exceptional sequences} (4) implies the two following things: firstly, $x_j^{(t_{\beta})} \not \simeq x_j^{(t_{\beta}+1)}$, which implies that $ x_j^{(t_{\beta})} \sqcup  \pi_j^{< t_{\beta}}  =  \pi_j^{\leq t_{\beta}}$; 
      secondly,  the $x_i^{(t_{\beta}+1)}  , \beta \in Q_1 \cap \mathrm{in}(i)$ are pairwise non isomorphic and also not isomorphic to $ x_i^{(0)} \simeq x_i$, so that $\pi_i^{\leq t_{\alpha}} \supseteq \{ x_i \} \sqcup  \{   x_i^{(t_{\beta}+1)} ,  \beta \in \overline{Q}_1 \cap \mathrm{in}(i) \}  $.
      Therefore we obtain
      \begin{align*}
      \pi_i^{\leq t_{\alpha}} \sqcup \bigsqcup_{\beta \in \overline{Q}_1 \cap \mathrm{in}(i)} \pi_{s(\beta)}^{< t_{\beta}} &\supseteq \{x_i\} \sqcup \bigsqcup_{\beta \in \overline{Q}_1 \cap \mathrm{in}(i)}  \left( x_i^{(t_{\beta}+1)} \sqcup  \pi_{s(\beta)}^{< t_{\beta}}     \right) \\
      & = \{x_i\} \sqcup \bigsqcup_{\beta \in \overline{Q}_1 \cap \mathrm{in}(i)} \left( x_{s(\beta)}^{(t_{\beta})} \sqcup  \pi_{s(\beta)}^{< t_{\beta}}   \right) = \{x_i\} \sqcup \bigsqcup_{\beta \in \overline{Q}_1 \cap \mathrm{in}(i)}  \pi_{s(\beta)}^{\leq t_{\beta}} . 
      \end{align*}
      We now move on to the rightmost inclusion. Still denoting $i := s(\alpha)$, we prove the desired inclusion by induction on the total dimension of the indecomposable projective $\mathbf{k}Q$-module at $i$. The initial step of the induction corresponds to the case where  $i$ is a sink in $Q$ (this number is then equal to $1$). In that case by construction there is a single arrow $\alpha : i \rightarrow \ast$ in $\overline{Q}_1$ and  we have $\barpi_{\alpha} = \pi_i^{\leq m} = \{x_i^{(t)} , 0 \leq t \leq m \}$. The fact that this set is included into $H(x_i)$ simply follows from Theorem~\ref{thm : exceptional sequences} (2) with $t=0$ and $t'$ running over $0, \ldots , m\}$. We now assume $i$ is not a sink in $Q$ and suppose the claimed inclusion holds for all vertices $j$ whose indecomposable projective has a total dimension strictly less than that of $i$. In particular, the induction assumption can be applied to all  vertices $j$ such that there is an arrow from $i$ to $j$ in $Q$. Therefore, applying Proposition~\ref{prop : mutation for objects in CQ} we get 
      \begin{align*}
         H(x_i) \sqcup H(\tau^{-1}x_i) & = \{x_i , \Sigma x_i\} \sqcup \bigsqcup_{j \rightarrow i}H(x_j) \sqcup \bigsqcup_{i \rightarrow k} H(\tau^{-1}x_k) \\
         &\supseteq \{x_i , \Sigma x_i\} \sqcup \bigsqcup_{j \rightarrow i}  \left( \pi_j \sqcup \bigsqcup_{\substack{ p \in \mathrm{in}(j) \\ \text{$p$ non lazy}}}  \pi_{s(p)}^{< t_p} \right)  \sqcup \bigsqcup_{i \rightarrow k} H(\tau^{-1}x_k)  \\
         &  \supseteq \{x_i , \Sigma x_i\} \sqcup \bigsqcup_{ \substack{p \in \mathrm{in}(i)\\ \text{$p$ non lazy}}} \pi_{s(p)}^{< t_p}  \sqcup \bigsqcup_{i \rightarrow k} H(\tau^{-1}x_k) . 
      \end{align*} 
    Thus to conclude it suffices to prove that 
    $$ \pi_i \sqcup H(\tau^{-1}x_i) \subseteq \{x_i , \Sigma x_i\}  \sqcup \bigsqcup_{i \rightarrow k}H(\tau^{-1}x_k) .   $$
    First of all we can note that $x_i \simeq x_i^{(0)} \in \pi_i$ and of course $x_i \notin H(\tau^{-1}x_i)$ hence the multiset on the left hand side contains exactly one copy of $x_i$. Similarly, $\Sigma x_i \notin \pi_i$ (as it does not even belong to $\mathrm{mod}\,(\mathbf{k}Q)$) and appears exactly once in $H(\tau^{-1}x_i)$ as $\Sigma x_i \simeq S \tau^{-1}x_i$. 
    So we now fix $y \in \pi \sqcup H(\tau^{-1}x_i)$ and assume  $y \not\simeq x_i$ and $y \not \simeq \Sigma x_i$. In particular the Auslander-Reiten formula (equivalently the fact that $S = \tau \Sigma$ is a Serre functor in $\DbRepQ$) yields
$$\mathrm{Ext}^1(\tau^{-1}x_i , y) \simeq D \mathrm{Hom}(y,x_i) = 0 .  $$
    Consequently  the number of times $y$ may appear in the multiset of the left hand side is 
    \begin{align*}
         \delta_{y \in \pi_i} +  \dim \mathrm{Hom}(\tau^{-1}x_i , y) &= \delta_{y \in \pi_i} + \langle \tau^{-1}x_i,y \rangle   = \delta_{y \in \pi_i} + \langle \tau^{-1}I_i , y \rangle + \sum_{i \rightarrow k} \langle \tau^{-1}x_k , y \rangle . 
    \end{align*}
    If $y \in \pi_i$, then $y \simeq x_i^{(t)}$ with $t>0$ because by assumption $y \not \simeq x_i \simeq x_i^{(0)}$. Hence Theorem~\ref{thm : exceptional sequences} (2) together with the Auslander-Reiten formula  implies 
    $$ \langle \tau^{-1}I_i , y \rangle  = - \langle y , I_i \rangle = - \langle x_i , y \rangle < 0.  $$
    Therefore the number of times $y$ appears on the left hand side is 
    $$ 1  + \langle \tau^{-1}I_i , y \rangle + \sum_{i \rightarrow k} \langle \tau^{-1}x_k , y \rangle  \leq \sum_{i \rightarrow k} \langle \tau^{-1}x_k , y \rangle \leq \sum_{i \rightarrow k} \dim \mathrm{Hom} \left(  \tau^{-1}x_k , y \right)  $$
    which is the number of times $y$ appears in the multiset $\bigsqcup_{i \rightarrow k} H(\tau^{-1}x_k)$. If now $y \notin \pi_i$ then $\delta_{y \in \pi_i} = 0$ while $\langle \tau^{-1}I_i , y \rangle  \leq \mathrm{\dim} \mathrm{Hom}(\tau^{-1}I_i,y) = 0$ as $y \in H(\tau^{-1}x_i)$ and $y \not \simeq \tau^{-1}I_i \simeq \tau^{-1}Sx_i = S \tau^{-1}x_i = \Sigma x_i$. The conclusion thus follows identically as in the case $y \in \pi_i$. This concludes the proof of the proposition. 

  \end{proof}
    
 We now give a couple of useful corollaries. Recall that  if $\gamma := \sum_i a_i \alpha_i$ with $a_i \geq 0$ for each $i \in I$, we denote by $Z_{\gamma}$ the multiset consisting of $a_i$ copies of $x_i$ for every $i \in I$.

  \begin{cor} \label{cor : multisets of projectives}
     Let $\alpha \in \overline{Q_1}$ and $i := s(\alpha)$. Then $(\barpi_{\alpha})_{\mid \mathrm{Ind}(\mathrm{Proj}\mathbf{k}Q) } = Z_{\boldsymbol{\dim}I_i}$.  
  \end{cor}

\begin{proof}
 For each $j \in I$, the indecomposable projective  $x_j$ appears exactly $p(j,i)$ times in the  hammock multiset $H(x_i)$. Consequently, $Z_{\boldsymbol{\dim}I_i} = H(x_i)_{\mid \mathrm{Ind}(\mathrm{Proj}\mathbf{k}Q)}$. Hence the inclusion    $(\barpi_{\alpha})_{\mid \mathrm{Ind}(\mathrm{Proj}\mathbf{k}Q) } \subseteq Z_{\boldsymbol{\dim}I_i}$ follows from the right most inclusion of Proposition~\ref{prop : multisets barpi}. The converse inclusion is proved by induction on the total dimension of the indecomposable injective $\mathbf{k}Q$ module at $i$. Namely,  if $i$ is a source in $Q$ then the inclusion is trivial because  $Z_{\boldsymbol{\dim}I_i} = \{x_i\}$ while it follows from~\eqref{eq : def of barpi} that  $\barpi_{\alpha} \supseteq \pi_i^{\leq t_{\alpha}} \ni x_i^{(0)} \simeq x_i$. If $i$ is not a source then we can write 
     $$ Z_{\boldsymbol{\dim}I_i} = \{x_i\} \sqcup \bigsqcup_{j \rightarrow i} Z_{\boldsymbol{\dim}I_j} =  \{x_i\} \sqcup \bigsqcup_{\beta \in Q_1 \cap \mathrm{in}(i)} Z_{\boldsymbol{\dim}I_{s(\beta)}} $$
     so the desired inclusion immediately follows from the induction assumption together with the leftmost inclusion of Proposition~\ref{prop : multisets barpi}. 
\end{proof}

 \begin{cor} \label{cor : unique decomposition}
    Let $\alpha_1, \ldots , \alpha_r, \beta_1, \ldots , \beta_s$ be arrows in $\overline{Q_1}$ and let $Z$ and $Z'$ be multisets having empty intsersection with $\mathrm{Ind}(\mathrm{Proj}\mathbf{k}Q)$. Assume $Z \sqcup \barpi_{\alpha_1} \sqcup \cdots \sqcup \barpi_{\alpha_r} = Z' \sqcup \barpi_{\beta_1} \sqcup \cdots \sqcup \barpi_{\beta_s} $. Then $Z = Z',r=s$ and $\{\alpha_1, \ldots , \alpha_r\} = \{\beta_1 , \ldots, \beta_s\}$.
 \end{cor}

  \begin{proof}
     Taking the intersection with $\mathrm{Ind}(\mathrm{Proj}(\mathbf{k}Q))$, Corollary~\ref{cor : multisets of projectives} yields $\sum_{k=1}^r \boldsymbol{\dim}I_{s(\alpha_k)} = \sum_{l=1}^s \boldsymbol{\dim}I_{s(\beta_l)}$. Thus it follows from the linear independence of the dimension vectors of the  indecomposable injective $\mathbf{k}Q$-modules that $r=s$ and (up to relabeling) $s(\alpha_k) = s(\beta_k) := i_k$ for each $1 \leq k \leq r$. Let $i$ (without loss of generality $i = i_1$) be a sink in the smallest (full) subquiver  $Q'$ of $Q$ containing $i_1, \ldots , i_r$. Then  it follows from the right most inclusion of Proposition~\ref{prop : multisets barpi} that  the projective $x_i$ appears exactly once in each $\barpi_{\alpha}$ such that $s(\alpha)=i$ and does  not appear in of the $\barpi_{\alpha}$ such that $s(\alpha) \neq i$. Thus we get $\barpi_{\alpha_1} \sqcup \cdots \sqcup \barpi_{\alpha_{r'}} = \barpi_{\beta_1} \sqcup \cdots \sqcup \barpi_{\beta_{r'}} $ where all the arrows involved here have the same source $i$. From~\eqref{eq : def of barpi}, this can be obviously rewritten as $\pi_i^{\leq t_{\alpha_1}} \sqcup \cdots \sqcup \pi_i^{\leq t_{\alpha_{r'}}} =\pi_i^{\leq t_{\beta_1}} \sqcup \cdots \sqcup \pi_i^{\leq t_{\beta_{r'}}}$.  Now it follows from the construction of the exceptional sequences given by Theorem~\ref{thm : exceptional sequences} that $\pi_i^{t_{\alpha}} $ is strictly contained in $\pi_i^{t_{\beta}}$ if $s(\alpha) = s(\beta) = i$ and $t_{\alpha} < t_{\beta}$. From this we see that $\max_{\alpha}t_{\alpha} = \max_{\beta}t_{\beta}$ so that this equality of multisets can be simplified. Continuing inductively we get (up to relabeling) $\alpha_1 = \beta_1 , \ldots , \alpha_{r'} = \beta_{r'}$. In other words we proved that $\alpha_k = \beta_k$ for every $k$ such that $s(\alpha_k) = s(\beta_k) = i$. Repeating the same arguments inductively, we conclude that $\alpha_k = \beta_k$ for all $1 \leq k \leq r$. The equality $Z=Z'$ then follows trivially. 
  \end{proof}

   \subsection{The category $\mathcal{R}_Q$}

 We now introduce the category that will serve as underlying category for the construction of exact chain complexes in the next section. Let $\tRQ$ denote the monoidal $\mathbf{k}$-category given by 
    $$ \tRQ := \bigoplus_{n \in \mathbb{Z}_{\geq 0}} \mathrm{mod}\, \left( \mathbf{k}Q^{\otimes n} \right) .  $$
The category $\tRQ$ is an abelian symmetric monoidal $\mathbf{k}$-category whose indecomposable objects are all the tensor products of the form $z_1 \otimes_{\mathbf{k}} \cdots \otimes_{\mathbf{k}} z_r$ where $z_1, \ldots , z_r$ are indecomposable objects in $\mathrm{mod}\, \mathbf{k}Q$. We outline on the fact that the tensor product is taken over the base field. We now fix an arbitrary total order on the set of isomorphism classes of indecomposable objects in $\mathrm{mod}\, \mathbf{k}Q$ and introduce a distinguished family $\{M_{\alpha} \}_{ \alpha \in Q_1}$ of indecomposable objects in $\tRQ$:
\begin{equation} \label{eq : def of Malpha}
    \forall \alpha \in Q_1, \enspace M_{\alpha} :=  \overrightarrow{ \bigotimes_{ z \in \barpi_{\alpha}}} z
\end{equation}
where the $\barpi_{\alpha} , \alpha \in Q_1$ are the multisets given by~\eqref{eq : def of barpi} and the arrow means that the tensor product is taken following the increasing order with respect to the chosen order $<$. Any object isomorphic to a tensor product  of the form $M_{\alpha_1} \otimes \cdots \otimes M_{\alpha_r}$ will be called a dominant object. On the other hand, an indecomposable object in $\tRQ$ of the form $z_1 \otimes \cdots  \otimes z_l$ will be called \textbf{neutral} if $\{z_1 , \ldots , z_l\} \cap \mathrm{Ind}(\mathrm{Proj \mathbf{k}Q}) = \emptyset$.  
We now define, for each $\alpha \in \overline{Q_1}$ a distinguished morphism $\eta_{\alpha}  \in \mathrm{Hom}_{\tRQ} \left( M_{\alpha} , \mu_{s(\alpha)}M_{\alpha} \right)$. Namely, writing $\pi^{\leq \alpha} := \{x_{i,1}, \ldots , x_{i,m} \}$, by Theorem~\ref{thm : exceptional sequences} we can choose a non-trivial morphism $f_{k, \alpha} \in \mathrm{Hom}_{\mathbf{k}Q}(x_{i,k},x_{i,k+1})$ for each $1 \leq k < m$ where $i := s(\alpha)$, as well as a non trivial morphism $f_{m, \alpha} \in \mathrm{Hom}_{\mathbf{k}Q}(x_{i,m},Sx_i)$. Then the pure tensor 
$$ \eta_{\alpha} := (f_{1, \alpha} \otimes \cdots \otimes f_{m, \alpha}) \otimes \cdots  $$
where the $\cdots$ stand for tensor products of identity morphisms, defines a non trivial morphism  in $\tRQ$
$$ \eta_{\alpha}^{M_{\alpha}} : \enspace  M_{\alpha} \simeq (x_{i,1} \otimes \cdots \otimes x_{i,m}) \otimes \cdots  \longrightarrow   (x_{i,2}  \otimes \cdots x_{i,m} \otimes Sx_i) \otimes  \cdots \simeq \mu_i M_{\alpha} . $$
 In other terms, using the diagrammatic description from Example~\ref{ex : morphisms in HQ}, the morphism $\eta_{\alpha}^{M_{\alpha}}$ can be represented as follows:
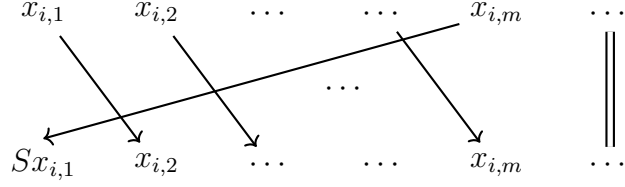
\begin{center}
    \begin{tikzpicture}
        
\node (a) at (0,2) {$x_{i,1}$};
\node (b) at (1.5,2) {$x_{i,2}$};
\node (c) at (3,2) {$\cdots $};
\node (c') at (4.5,2) {$\cdots$};
\node (d) at (6,2) {$x_{i,m}$};
\node (e) at (7.5,2) {$\cdots$};

\node (x) at (0,0) {$Sx_{i,1}$};
\node (y) at (1.5,0) {$x_{i,2}$};
\node (z) at (3,0) {$\cdots $};
\node (z') at (4.5,0) {$\cdots$};
\node (t) at (6,0) {$x_{i,m}$};
\node (u) at (7.5,0) {$\cdots$};

\node (plouf) at (4,1) {$\cdots$};

 \draw[thick,->] (a)--(y);
 \draw[thick,->] (b)--(z);
 \draw[thick,->] (c')--(t);
 \draw[thick,->] (d)--(x.north);
 \draw[thick,double distance = 2pt] (e)--(u);

    \end{tikzpicture}
\end{center}
 
We now define the category $\RQ$ as the smallest $\mathbf{k}$-linear monoidal subcategory of $\tRQ$ whose class of objects contains all the  objects isomorphic to $M_{\alpha} , \alpha \in \overline{Q}_1$ as well as all the neutral objects, and whose class of morphisms  contains the morphism $\eta_{\alpha}^{M_{\alpha}}$ for every $\alpha \in \overline{Q_1}$ as well as  all possible braiding isomorphisms. In particular, if $X$ is a neutral object then there is no (non zero, non iso) morphism of domain $X$ in $\RQ$. 

 \begin{lem} \label{lem : morphism eta}
   Let $\alpha \in \overline{Q}_1$, let $j := s(\alpha)$ and denote by  $M'_{\alpha}$  the codomain of $\eta_{\alpha}^{M_{\alpha}}$. Then we have $M'_{\alpha} \simeq X \otimes \bigotimes_{\beta \in Q_1 \cap \mathrm{in}(i)} M_{\beta}$ where $X$ is a neutral object. Moreover, we have  $\eta_{\beta}^{M'_{\alpha}} \circ \eta_{\alpha}^{M_{\alpha}} = 0$ for every $\beta \in Q_1 \cap \mathrm{in}(i)$.  
 \end{lem}

  \begin{proof}
    The first statement is a straightforward consequence of the leftmost inclusion of Proposition~\ref{prop : multisets barpi}. Concerning the second claim, let us denote by $i := s(\beta)$ and also $t := t_{\beta}$; then when writing the desired composition as a pure tensor, then by Theorem~\ref{thm : exceptional sequences} (3) and (4), one of the factors will be of the form $g_{l, \beta} \circ f_{k, \alpha}$ with 
    $$ x_j^{(t)} \xrightarrow[]{f_{k, \alpha}} x_j^{(t+1)} \simeq  x_i^{(t)} \xrightarrow[]{g_{l, \beta}} x_i^{(t+1)} $$
    which is zero by Theorem~\ref{thm : exceptional sequences} (3).
  \end{proof}

 \section{Exact complexes and standard $q$-characters}
 \label{sec : exact complexes}

 This section contains the most important results of this paper. Relying on the materials from Sections~\ref{sec : Euler characteristics and qcharacters} and~\ref{sec : the category RQ}, we show the existence for each indecomposable object $M$ in $\RQ$ of a chain complex $C_{\bullet}(M)$ satisfying strong homological properties that determine it uniquely up to homotopy. We then prove that (after a mild specialization), the Euler characteristics of the complexes  in the hammock category obtained by applying the functor $\mathcal{D}_Q$  to the $C_{\bullet}(M)$ coincide with the (truncated) $q$-characters of standard modules in the HL category $\Cxi^{(1)}$.

 \subsection{Definitions and classical properties}

 We fix a $\mathbf{k}$-linear category $\mathcal{A}$ together with a $\mathbf{k}$-linear subcategory  $\mathcal{C}$ of $\mathcal{A}$. Let $C_{\bullet}$ be a chain complex in $\mathcal{C}$, assumed to be of the form $$ C_{\bullet} :=  \quad  \cdots \longrightarrow 0 \longrightarrow C_0 \longrightarrow \cdots \longrightarrow C_m  \longrightarrow 0 \longrightarrow \cdots   \quad  $$
 For such a complex $C_{\bullet}$, the integer $m$ will be called the length of $C_{\bullet}$. 
 
 \begin{deftn} \label{def : exact complexes}
 We say that $C_{\bullet}$ is left  (resp. right) $\mathcal{C}$-exact  if $d_n$ is a weak cokernel of $d_{n-1}$ in $\mathcal{C}$ for every $n>0$ (resp. $d_{n-1}$ is a weak kernel of $d_n$ in $\mathcal{C}$ for every $n<m$). We say that $C_{\bullet}$ is $\mathcal{C}$-exact if it is both left and right $\mathcal{C}$-exact. 
 \end{deftn}

 Note in particular that with this definition, if $C_{\bullet}$ is left (resp. right)  $\mathcal{C}$-exact then $d_{m-1}$ (resp. $d_0$) is an epimorphism (resp. a monomorphism).
  The following is  a classical fact (cf. for example \cite{Jasso16}).  

 \begin{lem} \label{lem : exactness implies uniqueness}
    Let $C_{\bullet}$ and $D_{\bullet}$ be left $\mathcal{C}$-exact chain complexes. Assume moreover that $C_0 \simeq D_0$ and $d_0^{C_{\bullet}} = d_0^{D_{\bullet}}$. Then $C_{\bullet}$ and $D_{\bullet}$ are isomorphic up to homotopy. 
 \end{lem}

 \begin{deftn} \label{def : almost split complexes}
     We say that $C_{\bullet}$ is  $\mathcal{C}$-left almost split (resp.  $\mathcal{C}$-right almost split) if $C_{\bullet}$ is  $\mathcal{C}$-left (resp. right)  exact and moreover $C_{\bullet}$ does not split,  $C_0$ (resp. $C_m$) is indecomposable in $\mathcal{C}$ and  any morphism in $\mathcal{C}$ of domain $C_0$  (resp. of codomain $C_m$) that is not a split mono (resp. split epi) factors through $d_0$ (resp. $d_{m-1}$) in $\mathcal{C}$. We say that $C_{\bullet}$ is $\mathcal{C}$-almost split if it is both left and right $\mathcal{C}$-almost split. 
 \end{deftn}

 The following is a direct consequence of the above definitions.
 
 \begin{lem} \label{lem : morphism}
     Assume $C_{\bullet}$ is  $\mathcal{C}$-left exact and $D_{\bullet}$ is an arbitrary  chain complex in $\mathcal{C}$. Assume also that there are  non trivial morphisms $u : C_0 \rightarrow D_0$ and $u_1 : C_1 \rightarrow D_1$ in $\mathcal{C}$ such that $d_0^{D_{\bullet}} u_0 = u_1 d_0^{C_{\bullet}}$.  Then there is a non trivial morphism of chain complexes $u_{\bullet} : C_{\bullet} \rightarrow D_{\bullet}$ with $u_0 = u$.  
 \end{lem}

  \begin{proof}
     The composition $v_1 :=  d_1^{D_{\bullet}} u_1 $ is a morphism in $\mathcal{C}$ of domain $C_0$, and 
     $$v_1 d_0^{C_{\bullet}} = d_1^{D_{\bullet}} u_1 d_0^{C_{\bullet}} =  d_1^{D_{\bullet}}  d_0^{D_{\bullet}}u_0 = 0 .  $$
     Hence by left $\mathcal{C}$-exactness of $C_{\bullet}$, $v_1$ factors through $d_1^{C_{\bullet}}$ which gives a morphism $u_2 : C_2 \rightarrow D_2$ making the obvious square commute. Continuing inductively, we obtain the desired morphism $u_{\bullet} : C_{\bullet} \rightarrow D_{\bullet}$, and $u_{\bullet}$ is non trivial as $u_0$ and $u_1$ are. 
  \end{proof}



  \begin{lem} \label{lem : mapping cone}
      Let $C_{\bullet}$ and $D_{\bullet}$ be two complexes such that $C_{\bullet}$ is left $\mathcal{C}$-left exact and $D_{\bullet}$ is $\mathcal{C}$-left almost split. Assume  $u_{\bullet} : C_{\bullet} \longrightarrow D_{\bullet}$ is a morphism of chain complexes  such that $u_0$ does not factor through $d_0^{C_{\bullet}}$. Then the mapping cone of $u_{\bullet}[-1] $ is $\mathcal{C}$-left exact. 
  \end{lem}

   \begin{proof}
      Let $E_{\bullet}  := \mathrm{Cone}(u_{\bullet}[-1])$, let us fix $n>0$ and assume $f : E_n \rightarrow X$ is a morphism whose composition with $d_{n-1}^{E_{\bullet}}$ vanishes. We have  $E_n \simeq A_n \oplus B_{n-1} $  so writing $f$ as $(f_1 \enspace f_2)$ the condition $f d_{n-1}^{E_{\bullet}} = 0$ is equivalent to 
      $$   f_2 d_{n-2}^{D_{\bullet}} = 0 \qquad \text{and} \qquad f_1 d_{n-1}^{C_{\bullet}} + f_2 u_{n-1} = 0 .  $$
      If $n>1$, then by $\mathcal{C}$-left exactness of $D_{\bullet}$, the first condition implies that $f_2 = f'_2 d_{n-1}^{D_{\bullet}}$  so plugging into the second condition we get 
      $$  0 = f_1 d_{n-1}^{C_{\bullet}} + f'_2 d_{n-1}^{D_{\bullet}} u_{n-1} = f_1 d_{n-1}^{C_{\bullet}} + f'_2  u_n d_{n-1}^{C_{\bullet}}  = (f_1 + f'_2u_n) d_{n-1}^{C_{\bullet}} .  $$
      Thus by $\mathcal{C}$-left exactness of $C_{\bullet}$, we can write $f_1 + f'_2u_n = f'_1 d_n^{C_{\bullet}}$. All together we have 
      $$ (f_1 \enspace f_2) = (f'_1 \enspace -f'_2) \begin{pmatrix}
          d_n^{C_{\bullet}} & 0 \\ u_n & -d_{n-1}^{D_{\bullet}} 
      \end{pmatrix} 
      $$
      which shows that $f$ factors trough $d_n^{E_{\bullet}}$. Assume now $n=1$, so that the first condition becomes empty as $d_{-1}^{D_{\bullet}} = 0$. If $f_2$ was a split  monomorphism, then by definition $f_2$ would have a left inverse $h$, so that  the second condition could be  rewritten as $u_0 = -hf_1d_0^{C_{\bullet}}$  which contradicts the assumption that $u_0$ does not factor through $d_0^{C_{\bullet}}$. Therefore as $D_{\bullet}$ is  $\mathcal{C}$-left almost split, we can write $f_2 = f'_2 f_0^{D_{\bullet}}$. The conclusion then follows in an identical way as above.
   \end{proof} 

  We now assume that $\mathcal{A}$ is endowed with a symmetric monoidal structure $\otimes$, and that $\mathcal{C}$ is a monoidal subcategory of $\mathcal{A}$. If $M$ is an object in $\mathcal{C}$, we will often identify $M$ with the chain complex $ \cdots \rightarrow 0 \rightarrow M \rightarrow 0 \rightarrow \cdots$ where $M$ lies in homological degree zero. 

   \begin{deftn} \label{deftn : monoidal exactness}
      We say that $C_{\bullet}$ is purely $\mathcal{C}$-left (resp. right) exact  if $M \otimes C_{\bullet}$ is $\mathcal{C}$-left (resp. right) exact for any object $M$ in $\mathcal{C}$.  We say that $C_{\bullet}$ is purely $\mathcal{C}$-exact  if $ C_{\bullet}$ is both purely  $\mathcal{C}$-left and right exact. 
   \end{deftn}

  \subsection{Purely exact complexes in $\RQ$}

For each  indecomposable object $M \simeq  X \otimes M_{\alpha_1} \otimes \cdots \otimes M_{\alpha_r}$ in $\RQ$, with $r \geq 1$ and $\alpha_1, \ldots , \alpha_r \in \overline{Q}_1$ and $X$ is a neutral object, we define a dimension vector $\boldsymbol{d}(M)$ as well as a  positive integer $d(M)$ as 
$$ \boldsymbol{d}(M) := \sum_{l=1}^r \boldsymbol{\dim}  I_{s(\alpha_l)} \qquad d(M) := |\boldsymbol{d}(M)| . $$
Recall also that  if $\gamma := \sum_i a_i \alpha_i$ with $a_i \geq 0$ for each $i \in I$, then $Z_{\gamma}$ denotes the multiset consisting of $a_i$ copies of $x_i$ for every $i \in I$, i.e. $Z_{\gamma} := \bigsqcup_i \{x_i\}^{\sqcup a_i} $ .
 Let $M$ be a dominant object in $\RQ$. Write  $M := z_1 \otimes \cdots \otimes z_l$. Then Corollary~\ref{cor : multisets of projectives} implies that as multisets we have $Z_{\boldsymbol{d}(M)} \subseteq \{z_1, \ldots , z_l\}$. 
Consequently $\mu_{Z_{\boldsymbol{d}(M)}}M$ is an (indecomposable) object in $\RQ$. 
We can now state the main result of this section. 

  \begin{thm} \label{thm : main thm for standards}
     For every indecomposable object $M$ in $\RQ$, there exists a (unique up to homotopy) purely $\RQ$-exact almost split complex $C_{\bullet}(M)$ such that $C_0 \simeq M$.
  \end{thm}

The proof will be carried out by induction on $d(M)$. 
Let $M$ be an indecomposable object in $\RQ$ and assume the statement of Theorem~\ref{thm : main thm for standards} has been proved for all objects whose corresponding dimension vector is strictly less than that of $M$ for the usual order on the lattice $\bigoplus_i \mathbb{Z}\alpha_i$. Let us fix an isomorphism $M \simeq M_{\alpha_1} \otimes \cdots \otimes M_{\alpha_r}$ and consider $M' := M_{\alpha_2} \otimes \cdots \otimes M_{\alpha_r}$. Then obviously $d(M')< d(M)$ so the induction assumption provides us with a purely $\RQ$-exact almost split complex $C_{\bullet} := C_{\bullet}(M')$. In particular, $M_{\alpha_1} \otimes C_{\bullet}$ is $\RQ$-exact. On the other hand, by Lemma~\ref{lem : morphism eta} there is a non-trivial morphism 
$$\eta_{\alpha_1}^{M_{\alpha_1}} : M_{\alpha_1} \longrightarrow  X \otimes \bigotimes_{ \beta \in Q_1 \cap \mathrm{in}_Q(i)} M_{\beta}  \qquad \qquad  \text{where $i := s(\alpha_1)$}. $$
Let  us  consider the dominant object  $N := \left(\bigotimes_{ \beta \in Q_1 \cap \mathrm{in}_Q(i)} M_{\beta} \right) \otimes M_{\alpha_2} \otimes \cdots \otimes M_{\alpha_r} $. We have 
\begin{align*}
    d(N) &= \sum_{j \rightarrow i} \mathrm{dim tot}I_j + \sum_{l=2}^r \mathrm{dim tot}I_{s(\alpha_l)} \\
    &= \mathrm{\dim tot}I_i - \alpha_i + \sum_{l=2}^r \mathrm{dim tot}I_{s(\alpha_l)} = d(M)-\alpha_i  .
 \end{align*}
 Thus  by the induction assumption the complex $X \otimes D_{\bullet}$ is $\RQ$-exact and actually $\RQ$-almost split because there is no non isomorphism of domain $X$ in $\RQ$.  Let $u_0 := \eta \otimes \mathrm{id}_{M_{\alpha_2}} \otimes \cdots \otimes \mathrm{id}_{M_{\alpha_r}}$ and let us construct a morphism $u_1$ making the first square commute. By $\RQ$-left almost splitness of $D_{\bullet}$ and $C_{\bullet}$, the differentials $d_0^{M_{\alpha_1}} \otimes C_{\bullet}$ and  $d_0^{D_{\bullet}}$  can  respectively  be written as column morphisms $^t(g_1 \enspace \ldots  \enspace g_{r-1})$ and  $^t (f_1  \enspace \ldots \enspace  f_s  \enspace g'_1 \enspace  \ldots \enspace g'_{r-1})$ where 
 $$ g_l := \mathrm{id}_{M_{\alpha_1}} \otimes \mathrm{id}_{M_{\alpha_2}} \cdots \otimes \eta_{\alpha_{l+1}} \otimes \cdots \otimes \mathrm{id}_{M_{\alpha_r}},  $$
$$ f_k := \mathrm{id}_X  \otimes \cdots \otimes \eta_{\beta_k} \otimes  \cdots  \otimes \mathrm{id}_{M_{\alpha_2}} \otimes \cdots \otimes \mathrm{id}_{M_{\alpha_r}},  $$
$$ g'_l = \mathrm{id}_X \otimes \mathrm{id}_{M_{\beta_1}}  \otimes \cdots  \otimes \mathrm{id}_{M_{\beta_s}}  \otimes \mathrm{id}_{M_{\alpha_2}} \otimes \cdots \otimes \eta_{\alpha_{l+1}} \otimes \cdots \otimes \mathrm{id}_{M_{\alpha_r}}  $$
where the $\cdots $ are tensor products of identity morphisms, and the arrows $\beta_1, \ldots , \beta_s$ denote the arrows incoming into $i$ in $Q$. By Lemma~\ref{lem : morphism eta} we have 
 $$ \forall 1 \leq k \leq s, \enspace f_k \circ u_0 = \left(   \left( \cdots \otimes \eta_{\beta_k} \otimes \cdots \right) \circ \eta \right) \otimes \mathrm{id}_{M_{\alpha_2}} \otimes \cdots \otimes \mathrm{id}_{M_{\alpha_r}} = 0  . $$
 Hence the composition $d_0^{D_{\bullet}} \circ  u_0$ takes the form $^t (0 \enspace \ldots \enspace 0 \enspace h_1  \enspace \ldots  \enspace h_{r-1})$ where for each $1 \leq l \leq r-1$,  we have 
 \begin{align*}
 h_l &=  \left( \cdots \otimes \mathrm{id}_{M_{\alpha_2}} \otimes \cdots  \otimes \eta_{\alpha_l} \otimes \cdots  \otimes \mathrm{id}_{M_{\alpha_r}} \right) \circ  \left( \eta \otimes \mathrm{id}_{M_{\alpha_2}} \otimes \cdots \otimes \mathrm{id}_{M_{\alpha_r}} \right)  \\
 &= h'_l \circ \left( \cdots \otimes \mathrm{id}_{M_{\alpha_2}} \otimes \cdots  \otimes \eta_{\alpha_l} \otimes \cdots  \otimes \mathrm{id}_{M_{\alpha_r}} \right)
  \end{align*}
   (where $h'_l \neq 0$) by bifunctoriality of $\otimes$. 
 Thus we have constructed a morphism $u_1 : M_{\alpha_1} \otimes C_1 \longrightarrow D_1$ making the obvious square commute. Hence, as $M_{\alpha_1} \otimes C_{\bullet}$ is $\RQ$-left exact, Lemma~\ref{lem : morphism} yields a non trivial morphism $u_{\bullet} : M_{\alpha_1} \otimes C_{\bullet} \longrightarrow D_{\bullet}$. Moreover, it is clear from the definition of $\eta$ that for any object $T$ in $\RQ$,  $ \mathrm{id}_T \otimes u_0$ does not factor through $ \mathrm{id}_{T} \otimes \mathrm{id}_{M_{\alpha_1}} \otimes d_0^{C_{\bullet}} = \mathrm{id}_T \otimes d_0^{M_{\alpha_1} \otimes C_{\bullet}}$. Hence by Lemma~\ref{lem : mapping cone},  the complex $T \otimes \mathrm{Cone}(u_{\bullet}[-1])$ is $\RQ$-left exact i.e. $ E_{\bullet} := \mathrm{Cone}(u_{\bullet}[-1])$ is $\RQ$-purely left exact.  Moreover  we have $E_n \simeq (M_{\alpha_1} \otimes C_n )\oplus D_{n-1}$ so that $ E_n \simeq  0$ if $n<0$ and $E_0 \simeq M_{\alpha_1} \otimes C_0 \simeq M_{\alpha_1} \otimes M' \simeq M$. Thus to conclude it only remains to check that $E_{\bullet}$ is $\RQ$-left almost split. But this follows from the fact that by the mapping cone construction,
 $$ d_0^{E_{\bullet}} =  ^t ( \mathrm{id}_{M_{\alpha_1}} \otimes d_0^{C_{\bullet}} \quad u_0) = ^t(g_1 \enspace  g_2 \enspace \ldots  \enspace  g_r) $$
 where $ g_1 = \eta_{M_{\alpha_1}}   \otimes \mathrm{id}_{M_{\alpha_2}} \otimes \cdots \otimes  \mathrm{id}_{M_{\alpha_r}} $, so that any non split mono in $\RQ$  having  domain $M$ necessarily factors through $d_0^{E_{\bullet}}$. This shows the existence of the desired complex $C_{\bullet}(M)$, while the uniqueness up to homotopy follows from Lemma~\ref{lem : exactness implies uniqueness}. Note that all the arguments above were done with the notions of left exactness, left almost splitness etc but can be carried out in the same way with the dual notions replacing left by right.   This concludes the proof of Theorem~\ref{thm : main thm for standards}.

  \subsection{Standard $q$-characters as Euler characteristics}

  We conclude this section by  constructing chain complexes in the hammock category $\HQ$ from those obtained in $\RQ$, and computing their Euler characteristics. This is done by applying a functor $\mathcal{D}_Q$ that we now define, and whose existence is essentially a consequence of Proposition~\ref{prop : multisets barpi}. Let us consider $\RQ'$ the smallest full monoidal subcategory of $\RQ$ containing all the objects (isomorphic to objects) of the chain complexes $C_{\bullet}(M_{\alpha_1} \otimes \cdots \otimes M_{\alpha_l})$ for all collections $\alpha_1, \ldots , \alpha_r$ in $\overline{Q}_1$.

   \begin{prop} \label{prop : functor FQ}
     There is an additive monoidal functor $\mathcal{D}_Q : \RQ' \longrightarrow \HQ$ satisfying the following:
     \begin{enumerate} 
         \item For each arrow $\alpha \in \overline{Q}_1$, we have  $\mathcal{D}_Q(M_{\alpha}) = Y(x_{s(\alpha)})$. 
         \item  The functor $\mathcal{D}_Q$ commutes with Serre tiltings, i.e. If $M := X \otimes M_{\alpha_1} \otimes \cdots M_{\alpha_r}$ is an (indecomposable) object in $\RQ'$, then for each $1 \leq k \leq r$, the image under $\mathcal{D}_Q$ of the codomain of $\eta_{\alpha_k}^M$ is the Serre tilting $\mu_{s(\alpha_k)} \mathcal{D}_Q(M)$ in $\HQ$. 
     \end{enumerate} 
   \end{prop}

 \begin{proof}
     Let $\alpha \in \overline{Q}_1$, denote $i := s(\alpha)$ and let us write $M_{\alpha} := z_1 \otimes \cdots \otimes z_l$ where $z_1, \ldots , z_l$ are the elements of $\barpi_{\alpha}$. Let us denote by $M'_{\alpha}$ the codomain of $\eta_{\alpha}^{M_{\alpha}}$.  By construction, $ M'_{\alpha} \simeq z'_1 \otimes  \cdots \otimes z'_l$ where $\{z'_1 , \ldots , z'_l\} = \barpi_{\alpha} \setminus \{x_i\} \sqcup \{Sx_i\}$. Moreover  by Proposition~\ref{prop : multisets barpi}, we have $\barpi_{\alpha} \subseteq H(x_i)$. Combining these two facts, we see that $\eta_{\alpha}^{M_{\alpha}}$ can be naturally viewed as an element in $\mathrm{Hom}_{\HQ}(Y(x_i) , \mu_iY(x_i))$. In other words, we obtain an (injective) linear map 
     $$ \mathrm{Hom}_{\RQ}(M_{\alpha} , M'_{\alpha}) \longrightarrow \mathrm{Hom}_{\HQ} \left( Y(x_i) , \mu_iY(x_i) \right) .   $$
     Now this argument can be applied inductively to every indecomposable object in $\RQ'$, as these are all (up to isomorphism) of the form $X \otimes M_{\alpha_1} \otimes \cdots \otimes M_{\alpha_r}$. 
 \end{proof}

 We can now establish the second main result of this paper. 

  \begin{thm} \label{thm : char of standards as Euler char}
   Assume there exists a height function $\xi : I \rightarrow \mathbb{Z}$ compatible with $Q$, i.e. $\xi(j) = \xi(i)-1$ if there is an arrow $i \rightarrow j$ in $Q$. Let $\alpha_1, \ldots , \alpha_r$ in $\overline{Q}_1$
 and $M := M_{\alpha_1} \otimes \cdots \otimes M_{\alpha_r}$. Let also $\mathfrak{m}$ denote the dominant monomial given by $\mathfrak{m} := [\mathcal{D}_Q(M)]$. Then we have     
 $$ \chi \left( \mathcal{D}_Q(C_{\bullet}(M) \right)_{\mid F_i := 1} = \tchi_q(\Delta(\mathfrak{m})) .  $$
  \end{thm}

   \begin{proof}
       The proof goes by induction on $\boldsymbol{d}(M)$. The initial step corresponds to $\boldsymbol{d}(M) = \alpha_i$ which is possible only if $r=1$ i.e. $M \simeq M_{\alpha_1}$ with $i := s(\alpha_1)$ is a source in $Q$. In that case the complex $\mathcal{D}_Q(C_{\bullet}(M))$ is of the form  $Y(x_i) \longrightarrow \mu_iY(x_i) \longrightarrow 0 \longrightarrow \cdots $ so by Proposition~\ref{prop : mutation for objects in CQ} its Euler characteristics  is given by 
       $$ \chi \left( \mathcal{D}_Q(C_{\bullet}(M) \right) = Y_{x_i} - F_i Y_{\tau^{-1}x_i}^{-1} \prod_{i \rightarrow j}Y_{\tau^{-1}x_j} = Y_{x_i}(1 -F_iA_i^{-1}) .  $$
       On the other hand the truncated $q$-character of the fundamental representation $L(Y_{i, \xi(i)-2})$ is easily seen to be $Y_{i,\xi(i)-2}(1 + A_{i, \xi(i)-1}^{-1}) $. Hence, identifying $Y_{x_i} = Y_{\xi(i)-2} , A_i = A_{i, \xi(i)-1}$, the desired identity holds in that case. Now by construction if $\boldsymbol{d}(M)$ is not a simple root then $C_{\bullet}(M)$ is isomorphic to the (shifted) mapping cone of a morphism of chain complexes 
       $$ M_{\alpha_1} \otimes  C_{\bullet}(M') \enspace \longrightarrow \enspace  C_{\bullet}(M'_{\alpha_1} \otimes M') $$
       where $M' := M_{\alpha_2} \otimes \cdots \otimes M_{\alpha_r}$ and $M'_{\alpha_1}$ denotes the codomain of $\eta_{\alpha_1}^{M_{\alpha_1}}$ in $\RQ$. Now by Proposition~\ref{prop : functor FQ} we have $\mathcal{D}_Q(M'_{\alpha_1}) = \mu_i \mathcal{D}_Q(M_{\alpha_1})$ where $i := s(\alpha_1)$. Therefore  we get 
       $$  \chi \left( \mathcal{D}_Q(C_{\bullet}(M'_{\alpha_1} \otimes M' )) \right)   = F_i Y_{\tau^{-1}x_i}^{-1} \prod_{i \rightarrow j}Y_{\tau^{-1}x_j} \cdot    \chi \left( \mathcal{D}_Q(C_{\bullet}(N)) \right)  $$
       where $N \simeq M' \otimes \bigotimes_{\beta \in Q_1 \cap \mathrm{in}(i)}M_{\beta} $. 
       Using the induction assumption we obtain
    \begin{align*}
        Y_{\tau^{-1}x_i} \cdot \chi \left( \mathcal{D}_Q(C_{\bullet}(M)) \right) &= Y_{x_i}Y_{\tau^{-1}x_i} \cdot  \chi \left( \mathcal{D}_Q(C_{\bullet}(M')) \right)  - F_i  \prod_{i \rightarrow j}Y_{\tau^{-1}x_j} \cdot    \chi \left( \mathcal{D}_Q(C_{\bullet}(N)) \right)  \\
        &= Y_{x_i}Y_{\tau^{-1}x_i} \cdot \tchi_q(\Delta(\mathfrak{m'})) - F_i \prod_{i \rightarrow j}Y_{\tau^{-1}x_j} \cdot    \tchi_q (\Delta(\mathfrak{n}))
    \end{align*}
    where $\mathfrak{m}'$ and $\mathfrak{n}$ are the dominant monomials respectively given by $[\mathcal{D}_Q(M')]$ and $[\mathcal{D}_Q(N)]$. 
So in order to conclude, it suffices to show that the  standard truncated $q$-characters satisfy the following identity:
$$ \tchi_q(\Delta(\mathfrak{m})) = Y_{x_i} Y_{\tau^{-1}x_i} \cdot \tchi_q(\Delta(\mathfrak{m'})) + \prod_{i \rightarrow j}Y_{\tau^{-1}x_j} \cdot    \tchi_q (\Delta(\mathfrak{n})) .   $$
But given that $\mathfrak{m} = Y_{x_i}\mathfrak{m}'$ and  $\mathfrak{n} = \mathfrak{m}' \cdot \prod_{j \rightarrow i}Y_{x_j}$, this boils down to 
$$ Y_{\tau^{-1}x_i} \tchi_q(L(Y_{x_i})) = Y_{x_i}Y_{\tau^{-1}x_i} +\prod_{i \rightarrow j}Y_{\tau^{-1}x_j} \cdot \prod_{j \rightarrow i} \tchi_q(L(Y_{x_j})) $$
or in other terms 
$$ Y_{i, \xi(i)} \cdot \tchi_q(L(Y_{i, \xi(i)-2})) = Y_{i, \xi(i)-2}Y_{i, \xi(i)} \cdot + \prod_{i \rightarrow j}Y_{j, \xi(j)} \cdot \prod_{j \rightarrow i} \tchi_q(L(Y_{j, \xi(j)-2})) . $$
 This is simply a $T$-system,  as $\xi(j) = \xi(i)-1$ if $i \rightarrow j$ while $\xi(j)-2 = \xi(i)-1$ is $j \rightarrow i$. The theorem is proved. 
\end{proof}

 \subsection{A remark on $\mathcal{D}_Q$}
 \label{sec : functor DQ}

  The choice of notation $\mathcal{D}_Q$ for the above functor is made because $\mathcal{D}_Q$ should be viewed as a categorical pendant of the algebraic morphism  $\widetilde{D}_Q$ introduced in our joint work \cite{CasbiLi} with J.-R. Li (under a slightly modified form, and under its definitive form in \cite{CasbiLi2}). Though we will  not make any precise statement here, we would like to summarize the situation as follows: when passing to the Grothendieck group, the classes of objects in $\RQ$ can be written as products of dimension vectors of indecomposable objects in $\mathrm{mod}\,\mathbf{k}Q$, while as explained in Section~\ref{sec : hammock category}, the classes of objects in $\HQ$ can be written as Laurent monomials in $Y_{x_i},Y_{\tau^{-1}x_i}, i \in I$. Thus  the picture to keep in mind is the following
  $$ 
   \xymatrix{  
    \mathcal{K}^b(\RQ') \ar[rr]^{\mathcal{D}_Q} \ar[d]^{\chi} & {} &  \mathcal{K}^b(\HQ) \ar[d]^{\chi} \\
    \mathbf{k}[\mathfrak{t}^{*}]  \ar@{.>}[rd] & {} &  \YZ \ar[ld]_{\widetilde{D}_Q} \\
    {} & \mathbf{k}(\mathfrak{t}^{*}) & {} 
   }
  $$
  In view of Theorems~\ref{thm : main thm for standards} and~\ref{thm : char of standards as Euler char}, this shows that the complexes $C_{\bullet}(M)$ together with their images under  $\mathcal{D}_Q$ provide the relevant objects in order to obtain a uniform proof as well as (and most importantly) a conceptual interpretation of the identities obtained in \cite{CasbiLi2} (cf. the introduction above).

 \subsection{Example}

 We now illustrate the results above on a concrete example with  the quiver  $Q = 1 \xrightarrow[]{\alpha} 2 \xleftarrow[]{\beta} 3$ with the height function $\xi(2)=-2 , \xi(1) = \xi(3) = -1$ and $M := M_{\beta}$. The exceptional sequences given by Theorem~\ref{thm : exceptional sequences} are as follows 
 $$ \xymatrixrowsep{2pt} \xymatrix{
 P_2 \ar@{.>}[r] & P_1 \ar@{.>}[r] &  I_2  \\
 P_3  \ar@{.>}[r]  & I_2 \ar@{.>}[r]  & I_3 \\
 P_1 \ar@{.>}[r] & I_1 \ar@{.}[r] & I_1 
 }
 $$
so that $t_{\alpha} = 0$ and $t_{\beta} = 1$. We have $\barpi_{\beta} = \pi_2^{\leq 1} \sqcup \pi_1^{<0} \sqcup \pi_3^{<1}  = \{P_2,P_1\} \sqcup \{P_3\} = \{P_2,P_1,P_3\}$. Figure~\ref{fig : braid diagram} shows the chain complex $C_{\bullet}(M_{\beta})$ as well as the Euler characteristics of its image under $\mathcal{D}_Q$ (written under the form of a graph where each arrow carries a color $i$ representing multiplication by $A_i^{-1}$). This graph is well-known to be that of the truncated $q$-character of the fundamental representation $L(Y_{2,-2})$. Note that on this diagrammatic picture, it is easy to check that that the sequence on the left of Figure~\ref{fig : braid diagram} is indeed a chain complex.

 \begin{center}

 \begin{figure}

    \begin{tikzpicture}


 \node (a1) at (-0.8,0) {$P_2$};
 \node (a5) at (-0.4,0) {$\otimes$};
\node (a2) at (0,0) {$P_1$};
\node (a6) at (0.4,0) {$\otimes$};
\node (a3) at (0.8,0) {$P_3$};

\node (b1) at (-0.8,-2) {$I_2$}; 
\node (b5) at (-0.4,-2) {$\otimes$};
\node (b6) at (0.4,-2) {$\otimes$};
\node (b2) at (0,-2) {$P_1$};
\node (b3) at (0.8,-2) {$P_3$};

\node (c1) at (-2.8,-4) {$I_2 $};
\node (c5) at (-2.4,-4) {$\otimes$};
\node (c6) at (-1.6,-4) {$\otimes$};
\node (c2) at (-2,-4) {$I_1$};
\node (c3) at (-1.2,-4) {$P_3$};

\node (d1) at (1.2,-4) {$I_2$}; 
\node (d5) at (2.4,-4) {$\otimes$};
\node (d6) at (1.6,-4) {$\otimes$};
\node (d2) at (2,-4) {$P_1$};
\node (d3) at (2.8,-4) {$I_3$};

\node (e1) at (-0.8,-6) {$I_2$};
\node (e5) at (-0.4,-6) {$\otimes$};
\node (e6) at (0.4,-6) {$\otimes$};
\node (e2) at (0,-6) {$I_1$};
\node (e3) at (0.8,-6) {$I_3$};

 \node (plouf) at (0,-4) {$\oplus$};

\draw[thick,->] (a1)--(b2);
\draw[thick,->] (a2)--(b1);
\draw[thick,->] (b2) to[out=-90,in=90] (c2);
\draw[thick,->] (b3) to[out=-90,in=90] (d1);
\draw[thick,->] (b1) to[out=-90,in=90] (d3);
\draw[thick,->] (c3) to[out=-90,in=90] (e1);
\draw[thick,->] (c1) to[out=-90,in=90] (e3);
\draw[thick,->] (d2) to[out=-90,in=90] (e2);


\node (x) at (8,0) {$Y_{2,-2}$};
\node (y) at (8,-2) {$Y_{2,0}^{-1}Y_{1,-1}Y_{3,-1}$};
\node (z) at (6,-4) {$Y_{1,1}^{-1}Y_{3,-1}$}; 
\node (t) at (10,-4) {$Y_{3,1}^{-1}Y_{1,-1}$}; 
\node (u) at (8,-6) {$Y_{2,2}Y_{1,1}^{-1}Y_{3,1}^{-1}$};

\draw[thick] (x)--(y) node[pos=0.5,left,font=\tiny] {$2$} ;

\draw[thick] (y)--(z) node[pos=0.6,above,font=\tiny] {$1$} ;

\draw[thick] (y)--(t) node[pos=0.6,above,font=\tiny] {$3$} ;

\draw[thick] (z)--(u) node[pos=0.4,below,font=\tiny] {$3$} ;

\draw[thick] (t)--(u) node[pos=0.4,below,font=\tiny] {$1$} ;

    \end{tikzpicture}

 \caption{The chain complex $C_{\bullet}(M_{\beta})$ (on the left) and the graph displaying the classes in $K_0(\HQ)$ of the images of each object under the functor $\mathcal{D}_Q$ (on the right).}
\label{fig : braid diagram}

 \end{figure}
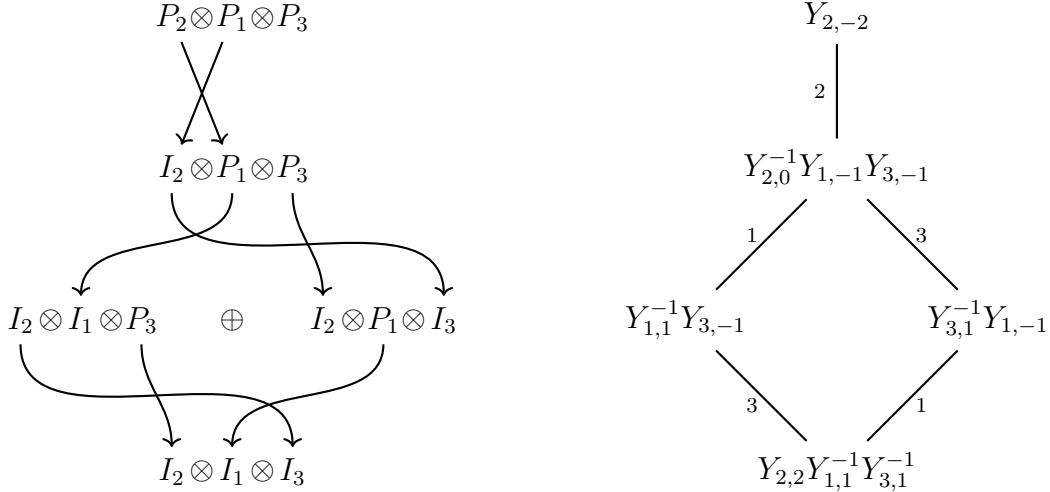

 \end{center}

\section{Cluster characters as Euler characteristics}
 \label{sec : Euler characteristics and qcharacters}

 The purpose of this section is to show that the methods proposed in this paper can be used not only to recover the $q$-characters of standard modules in the HL categories, but also (at least some) simple modules. Though we will leave a more systematic investigation of this problem for future works, we show here an instance of this when $Q$ is an arbitrary orientation of a type $A_n$ Dynkin graph. Let us choose a height function $\xi I \rightarrow \mathbb{Z}$ adapted to $Q$.  In view of the monoidal categorification results due to Brito-Chari \cite{BC} and Kashiwara-Kim-Oh-Park \cite{KKOP1}, we know that all (the classes of) simple modules in the HL category $\Cxi^{(1)}$  are cluster monomials, and  that their truncated $q$-characters can be identified (up to a mild change of variables) with their cluster expansion with respect to an appropriately chosen initial seed.  Therefore in what follows we will essentially work exclusively using the language of cluster algebras and will not rely on any representation-theoretic tool throughout the proofs. Let us begin with a brief reminder of the combinatorics of finite type cluster algebras.

 \subsection{Identities between $F$-polynomials in  cluster algebras of type $A_n$}
  \label{sec : reminders AQ}

 Given a Dynkin quiver $Q$ (in this section, we will mostly focus on the case where $Q$ is of type $A_n, n\geq 1$), we will be denoting by $\AQ$ the cluster algebra with initial seed
 $ ((u_1, \ldots , u_n),  Q)$.
  It follows from Fomin-Zelevinsky's results \cite{FZ2} that $\AQ$ is a finite type cluster algebra, i.e. $\AQ$ contains finitely many  clusters (equivalently finitely many cluster variables). Consequently, the non-frozen cluster variables in $\AQ$ are in bijection with the set of almost positive roots which is defined as $\Delta_+ \cup \Pi_-$. We will denote by $u[\beta]$ the cluster variable corresponding to $\beta$ under this bijection, for each almost positive root $\beta$. In particular, the cluster variables $u[- \alpha_i] , i \in I$  respectively coincide with the cluster variables $u_1, \ldots , u_n$ belonging to the initial cluster. Fomin-Zelevinsky \cite{FZ4} proved that the cluster expansion of $u[\beta]$  with respect to the initial cluster can be written under the form 
  $$ x[\beta] = \mathbf{x}^{g_{\beta}} \cdot F[\beta](\widehat{y_1}, \ldots , \widehat{y_n}) $$
  where $g_{\beta} := (g_1, \ldots , g_n)$ is an integer-valued vector called the $g$-vector (here $\mathbf{u}^g := u_1^{g_1} \cdots u_n^{g_n}$) and $F[\beta]$ is a polynomial in the variables $\widehat{y}_1 , \ldots , \widehat{y}_n$ which are defined by 
  $$ \widehat{y}_i = \prod_{i \rightarrow j} u_j \cdot \prod_{i \leftarrow j} u_j^{-1} .  $$
The polynomial $F[\beta]$ is called the $F$-polynomial of $u[\beta]$.  It is known that $F[\beta]$ has constant term $1$, and that its highest degree term is given by $\mathbf{\widehat{y}}^{\beta} := \widehat{y}_1^{a_1} \cdots \widehat{y}_n^{a_n}$ where $\beta := \sum_i a_i \alpha_i$.

 \begin{lem} \label{lem : cluster mutation}
  Assume $Q$ is an arbitrary orientation of a type $A_n$ Dynkin diagram, with $n \geq 1$. Let  $\beta \in \Delta_+$ and $i$ be a sink in $\Qbeta$. Then we have $ F[\beta]  = F[\beta - \boldsymbol{\dim}I_i^{\beta}] + \widehat{y_i} F[\beta-\alpha_i] $. 
 \end{lem}

   \begin{proof}
   Let us choose a triangulation of a  regular $2n$-gon corresponding to the chosen orientation $Q$. The fact that $i$ belongs to $(Q_{\beta})_0$  tells us that that the diagonal corresponding to the negative simple root $-\alpha_i$ crosses the diagonal corresponding to the positive root $\beta$. Hence they are not compatible, and thus the Ptolemy relation 
   $$ u[\beta]u[-\alpha_i] = M_1 + M_2 $$
   is an exchange relation corresponding to a cluster mutation in $\AQ$. Let us now describe $M_1$ and $M_2$ explicitly. Let us write $\beta$ as a segment $[k,l]$ with $k \leq i \leq l$. Also let us denote by $(a,b)$ (resp. $(c,d)$) the diagonals respectively corresponding to $u[\beta]$ and $u[-\alpha_i]$.   It is convenient to distinguish two cases, whether $i$ is strictly inside $[k,l]$ or on the border. 

    In the former case, as $i$ is a sink in $\Qbeta$, the situation looks like 
    $$ k   \cdots \leftarrow  k' \rightarrow  \cdots  \rightarrow i  \leftarrow \cdots \leftarrow l' \rightarrow \cdots l .  $$
    We see that (up to renaming) $(a,d)$ crosses the same diagonals corresponding to the negative simple roots  as $(a,b)$  except the one corresponding to $u[-\alpha_i]$ (i.e. $(c,d)$), while $(b,c)$ is an edge of the $2n$-gon. Thus $u[a,d]u[b,c] = u[\beta-\alpha_i]$. On the other hand, we see that $(b,d)$ (resp.  $(a,c)$)  crosses the diagonals corresponding to $u[\alpha_p] , k \leq p <k' $ (resp. $l'< p \leq l $) so that $u[b,d]u[a,c] = u[\alpha_{k} + \cdots \alpha_{k'-1}]u[\alpha_{l'+1} + \cdots \alpha_l]$. But these two roots being obviously compatible, this product is a cluster monomial, whose denominator vector is given by 
    $$ \beta - (\alpha_{k'} + \cdots + \alpha_{l'}) = \beta - \boldsymbol{\dim}I_i^{\beta} . $$
    In the case where $i$ is on the border of $[k,l]$ (without loss of generality $i=l$) the situation looks like 
    $$ k \cdots \leftarrow k' \rightarrow \cdots \rightarrow i .  $$
    We can repeat the same arguments as in the previous case. The only difference is if $i$ is not a sink in $Q$, i.e. there is an arrow $i \rightarrow i+1$ in $Q$; then $M_1$ will be given by $u[-\alpha_{i+1}]u[\beta-\alpha_i]$ (which is a cluster monomial as $-\alpha_{i+1}$ and $\beta-\alpha_i$ are compatible). Summarizing the two cases, this exchange relation can be written uniformly as
    \begin{equation} \label{eq : cluster mutation}
          u[\beta]u[-\alpha_i] = u[\beta'] + \prod_{i \rightarrow j}u_j \cdot u[\beta-\alpha_i]
    \end{equation}
    where $\beta' := \beta - \boldsymbol{\dim}I_i^{\beta}$ and the second term on the right hand side is a cluster monomial because $i$ is a sink in $\Qbeta$. Let us write $u[\gamma] = \mathbf{u}^{g_{\gamma}} F[\gamma]$ for each root $\gamma$. As $\beta-\alpha_i \geq  \beta'$, the term of highest degree in the $\widehat{y}_l$ on the right hand side comes from $F[\beta-\alpha_i]$. Hence comparing with that coming  from the left hand side, we get 
    $$ \mathbf{u}^{g_{\beta}} \mathbf{\widehat{y}}^{\beta} u_i = \prod_{i \rightarrow j} u_j \cdot \mathbf{u}^{g_{\beta-\alpha_i}} \mathbf{\widehat{y}}^{\beta-\alpha_i} $$
    which yields
    $$ \mathbf{u}^{g_{\beta}}  u_i \widehat{y_i} = \prod_{i \rightarrow j} u_j \cdot \mathbf{u}^{g_{\beta-\alpha_i}} .  $$
   Thus, dividing~\eqref{eq : cluster mutation} by $\mathbf{u}^{g_{\beta}}  u_i $ we get 
   $$  F[\beta] = \frac{\mathbf{u}^{g_{\beta'}}}{\mathbf{u}^{g_{\beta}}u_i} F[\beta'] + \widehat{y_i}F[\beta-\alpha_i] . $$
   Now recalling that the constant term of all $F$-polynomials are $1$, this identity also implies 
   $$  1 = \frac{\mathbf{u}^{g_{\beta'}}}{\mathbf{u}^{g_{\beta}}u_i} $$
   so in conclusion we obtain 
   $$ F[\beta] =  F[\beta'] + \widehat{y_i}F[\beta-\alpha_i] .  $$
   \end{proof}

  \subsection{Combinatorial preliminaries}
   \label{sec : weights of dom functions}


This subsection provides necessary technicalities that will be needed to prove the last main result of this paper. Note that all the results in this subsection hold for any acyclic quiver without multiple edges. 
Recall that for any $\gamma := \sum_{i \in I} a_i \alpha_i$ with $a_i \geq 0$ for all $i \in I$, we will denote by $Z_{\gamma}$ the multiset containing $a_i$ copies of $x_i$ for each $i \in I$, i.e. $ Z_{\gamma} := \bigsqcup_{i \in I} \{ x_i\}^{\sqcup a_i}$. We now associate to any quasi-additive function $h$ on $\ZQ$ an element $\bomega(h) \in \bigoplus_i \mathbb{Z}_{\geq 0} \alpha_i$ as follows. For each vertex $i \in \overline{Q_0}$, we define a non-negative integer $a_i$ by induction on the total dimension of the indecomposable projective $\mathbf{k}Q$-module at $i$. Namely, we set $a_{\ast}  := 0$ and 
 $$ \forall i \in Q_0, \enspace  a_i := \max \left(  0 ,  \widetilde{h}(x_i)-\widetilde{h}(\tau^{-1}x_i) + \sum_{i \rightarrow j} a_j  \right) .  $$
 We then set 
  $$ \bomega(h) := \sum_{i \in I} a_i \alpha_i  \qquad  Z(h) :=  Z_{\bomega(h)} .  $$
  We will sometimes refer to $\bomega(h)$ as the weight of $h$. 
Recall the notations $P_i^{\beta}, I_i^{\beta}$ from Section~\ref{sec : reminders on AR}. 

  \begin{lem} \label{lem : weight first lemma}
 Let $h$ be a dominant quasi-additive function on $\ZQz$. Then we have 
 $$  \forall i \in (\Qbeta)_0 , \enspace \beta - \boldsymbol{\dim}I_i^{\beta} \leq \bomega \left( h +  h_{\tau^{-1}x_i} \right) \leq \beta - \alpha_i  . $$ 
 Moreover the left most inequality is an equality if there is at most one oriented path between any two vertices in $Q$. 
  \end{lem}

 \begin{proof}
  Let us  denote  $\beta = \sum_i a_i \alpha_i$ and  $c_k := \widetilde{h}(x_k), d_k := \widetilde{h}(\tau^{-1}x_k)$ for every $k \in I$. Let $h' := h +  h_{\tau^{-1}x_i}$ and $\bomega(h') := \sum_k a'_k \alpha_k$. Note that $\widetilde{h'}(x_k) = c_k$ and $\widetilde{h'}(\tau^{-1}x_k) = d_k + \delta_{i,k}$ for all $k \in I$. We prove by induction on $r_Q(k)$ that $   a_k - p_{\Qbeta}(k,i)\leq  a'_k \leq a_k - \delta_{i,k}$ for all $k \in I$.  If $k$ is a sink in $Q$ then obviously $p_{\Qbeta}(k,i) = \delta_{i,k}$ so we need to prove that $a'_k =a_k-\delta_{i,k}$.  By definition we have $a'_k = \max(0,c_k-d_k-\delta_{i,k})$. If $k \neq i$ then we get $a'_k=a_k$. If $k=i$ then as $i \in (\Qbeta)_0$ we have $0<a_i=c_i-d_i$ (as $i=k$ is a sink)  and hence $a'_k=a_i-1$. Thus the claim holds in that case. Let us fix $k \in I$ and assume the desired inequalities hold for all $j$ such that $r_Q(j)<r_Q(k)$. Then we have 
   \begin{align*}
      a'_k &= \max \left( 0 , c_k-d_k-\delta_{i,k} + \sum_{k \rightarrow j} a'_j \right) . 
   \end{align*}
   If $k \notin (\Qbeta)_0$ then in particular $k \neq i$ and moreover $a_k=0$ so $c_k-d_k + \sum_{k \rightarrow j}a_j \leq 0$. Hence  using the induction assumption we get that 
   $$ c_k-d_k + \sum_{k \rightarrow j} a'_j  \leq    c_k-d_k + \sum_{k \rightarrow j} (a_j - \delta_{i,j}) \leq   c_k-d_k + \sum_{k \rightarrow j} a_j \leq 0  $$
   and hence $a'_k=0$. So we get $ a_k- p_{\Qbeta}(k,i) =  -p_{\Qbeta}(k,i) =\leq 0 = a'_k \leq a_k = a_k-\delta_{i,k}$ as desired. If on the other hand $k \in (\Qbeta)_0$ then $a_k>0$ so in particular $a_k- \delta_{i,k} \geq 0$ and moreover using the induction assumption we have  
   $$ a_k- \delta_{i,k} = c_k-d_k - \delta_{i,k}  + \sum_{k \rightarrow j} a_j \geq c_k-d_k - \delta_{i,k}  + \sum_{k \rightarrow j} a'_j .   $$
   Therefore $a_k- \delta_{i,k} \geq a'_k$. On the other hand,
   \begin{align*}
       a'_k &\geq c_k-d_k-\delta_{i,k}+ \sum_{k \rightarrow j}a'_j  \\
        &\geq c_k-d_k-\delta_{i,k} + \sum_{k \rightarrow j}(a_j-p_{\Qbeta}(j,i))  \\
       &= a_k - \delta_{i,k} - \sum_{k \rightarrow j}p_{\Qbeta}(j,i) = a_k-p_{\Qbeta}(k,i) . 
   \end{align*}
   Thus we get the desired inequalities. If now we furthermore assume that there is at most one oriented path between any two vertices in $Q$, then in the proof above we can use the fact that if $j \in (\Qbeta)_0$ i.e. $a_j>0$ then $a_j-p_{\Qbeta}(j,i) \geq 0$ so in the induction we can plug-in $a'_j = a_j-p_{\Qbeta}(j,i)$ and conclude using the fact that $p_{\Qbeta}(k,i) = \delta_{i,k} + \sum_{k \rightarrow j} p_{\Qbeta}(j,i)$.
   This concludes the proof of the lemma. 
 \end{proof}

 \begin{lem} \label{lem : weight second lemma}
    Let $h$ be a dominant quasi-additive function, let $\beta := \bomega(h)$ and let $i$ be a sink in $\Qbeta$. Then we have $ \bomega(h + h_{\tau^{-1}x_i} - \delta_{x_i}) = \beta - \alpha_i$.
 \end{lem}

 \begin{proof}
 Denote $c_k := \widetilde{h}(x_k)$ and $d_k := \widetilde{h}(\tau^{-1}x_k)$ for every $k \in I$. Let also denote $h' := h + h_{\tau^{-1}x_i} - \delta_{x_i}$.  As $i$ is a sink in $\Qbeta$,  one can check  by a straightforward induction  that  $c_j=d_j=0$ for all $j$ such that $i \rightsquigarrow j$. Thus we have $0<a_i=c_i-d_i$ so that $c_i > d_i \geq 0$. Consequently, we can write $h = h_{x_i} + f$ and thus applying  Lemma~\ref{lem : mutation for functions} we get 
 $$ h' = f + \sum_{k \rightarrow i} h_{x_k} + \sum_{i \rightarrow k} h_{\tau^{-1}x_k}  $$
 so that 
 where 
 $$ c'_k := \widetilde{h'}(x_k) =   c_k - \delta_{i,k} + \delta_{i \leftarrow k} \qquad d'_k := \widetilde{h'}(\tau^{-1}x_k) =    d_k+\delta_{i \rightarrow k}   $$
  as  there is an arrow $x_i \rightarrow x_k$ (resp. $x_i \rightarrow \tau^{-1}x_k$) in $\ZQ$ if $i \leftarrow k$ (resp. $i \rightarrow k$) in $Q$. Note that $h'$ is dominant. We set $\bomega(h') := \sum_k a'_k \alpha_k$ and we prove by induction on $r_Q(k)$ that $a'_k=a_k- \delta_{i,k}$ for all $k \in I$. If $r_Q(k)=0$ i.e. $k$ is a sink in $Q$ then 
  \begin{align*}
      a'_k & =c'_k=c_k-\delta_{i,k} + \delta_{i \leftarrow k} = c_k- \delta_{i,k} = a_k- \delta_{i,k}
  \end{align*}
  as $k$ is a sink in $Q$. If $k \in I$ is such that  the claim holds for all $j$ with $r_Q(j)<r_Q(k)$ then  
\begin{align*}
  a'_k &= \max \left( 0, c_k- d_k   - \delta_{i,k} + \delta_{i \leftarrow k} - \delta_{i \rightarrow k} + \sum_{k \rightarrow j}  (a_j- \delta_{i,j})    \right)   \\
  &= \max \left(  0, c_k-d_k - \delta_{i,k} - \delta_{i \rightarrow k} + \sum_{k \rightarrow j}a_j \right) 
\end{align*}
 so that $a'_k=a_k- \delta_{i,k}$ if there is no arrow from $i$ to $k$ in $Q$. If on the contrary we have $i \rightarrow k$ then in particular $k \notin (\Qbeta)_0$ as $i$ is a sink in $\Qbeta$, and hence necessarily $a_k=0$ which means that $c_k-d_k-\sum_{k \rightarrow j}a_j \leq 0$ and thus we get $a'_k=0$. Thus we have $\bomega(h') = \beta - \alpha_i$ as desired. 
 \end{proof}

  \subsection{Cluster variables of type $A_n$ as Euler characteristics}


\begin{thm} \label{thm : main thm simples}
 Let $\beta \in \Delta_+$ and $h$ be a dominant quasi-additive function such that $\bomega(h) = \beta$. Then there exists a chain complex $C_{\bullet}(h) \in \mathcal{K}^b(\HQ)$ of length $|\beta|$ 
 whose renormalized Euler characteristics coincides with the $F$-polynomial $F[\beta]$. 
 \end{thm}

  \begin{proof}
      Let us first outline the strategy of the proof. Recall the chain complex $C_{\bullet}(N)$ given by Theorem~\ref{thm : main thm for standards} for each object $N$ in $\RQ'$.  We begin by establishing the following claim. Let $h := \sum_k (c_k h_{x_k} + d_k h_{\tau^{-1}x_k}$) be a dominant quasi-additive function. Then for every  $M \in \RQ'$ such that $\mathcal{D}_Q(M) = Y(x_1)^{\otimes c_1} \otimes_{\HQ} \cdots \otimes_{\HQ} Y(x_n)^{\otimes c_n}$, there exists a chain complex $C_{\bullet}(M,h)$ such that $C_0(M,h) = M$  and such that for any $Z \subseteq \bigsqcup_k \{x_k\}^{\sqcup \min(c_k,d_k)}$, there is a morphism $C_{\bullet}(M,h) \longrightarrow C_{\bullet}(N)[1-d]$ where $N$ is such that $\mathcal{D}_Q(N) = \mu_Z \mathcal{D}_Q(M)$ and $d := \sharp Z$. 

       Assuming the claim holds,  we will then show that the image of $C_{\bullet}(M,h)$ under the functor $\mathcal{D}_Q$  satisfies all the desired properties.

\bigskip

 Let us first prove the claim. We proceed by induction on both $d(M)$ and $\beta := \bomega(h)$. If $\beta = 0$ then it suffices to define $C_{\bullet}(M,h) := M \rightarrow 0 \rightarrow \cdots $. If $\beta \neq 0$, then let  $i$ be a sink in $\Qbeta$.  Note that as $i$ is a sink in $\Qbeta$, we have $0 < a_i = c_i-d_i$. Then denoting $h'  := h+h_{\tau^{-1}x_i}$, we still have $h' \sim f$ and $\min(c'_i,d'_i) \geq 1$. As $\bomega(h')< \bomega(h)$ by  Lemma~\ref{lem : weight first lemma} the induction assumption with $h'$ and $Z := \{x_i\}$, yields a morphism  $C_{\bullet}(M,h') \longrightarrow C_{\bullet}(N)$ where  as above $N \in \RQ'$ is such that $\mathcal{D}_Q(N) = \mu_i \mathcal{D}_Q(M)$. Note that $d(N) = d(M)-\alpha_i$ so we can apply the induction assumption with $N$ and  $g := h'- \delta_{x_i}$. Indeed, $g  \sim f -\delta_{x_i} $ while by $\mathcal{D}_Q(N) = \mu_iY(f) $.  The induction assumption provides a chain complex $C_{\bullet}(N,g)$ whose degree $0$ object is $N$. Thus Lemma~\ref{lem : morphism}  yields a morphism 
 $C_{\bullet}(N) \longrightarrow C_{\bullet}(N,g)$.  We then construct $C_{\bullet}(M,h)$ as the mapping cone of the composition of the two previous morphisms:
$$  C_{\bullet}(M,h) := \mathrm{Cone} \left(  C_{\bullet}(M,h')[-1]  \enspace \longrightarrow \enspace  C_{\bullet}(N,g)[-1] \right) . $$
 It remains to check that $C_{\bullet}(M,h)$ satisfies the property stated in the claim. Obviously, $C_0(M,h) = C_0(M,h') = M $. 
 Thus, let us take  $Z \subseteq \bigsqcup_k \{x_k\}^{\sqcup \min(c_k,d_k)}$ arbitrary. If $d_i \geq c_i$ then  the induction assumption  gives us a morphism 
 $$ C_{\bullet}(M,h') \longrightarrow C_{\bullet}(N)$$
 so composing with the morphism $C_{\bullet}(M,h) \longrightarrow C_{\bullet}(M,h')$ coming from the definition of $C_{\bullet}(M,h)$ we get the desired conclusion. If on the other hand  $d_i<c_i$  then  the induction assumption applied with $h'$ and $Z' := Z \sqcup \{x_i\}$ gives a morphism 
 $$C_{\bullet}(M,h') \longrightarrow C_{\bullet}(N')[1-(d+1)] = C_{\bullet}(N')[-d] $$
 where $N' \in \RQ'$ is such that $\mathcal{D}_Q(N') = \mu_{Z'}\mathcal{D}_Q(M)$. But looking at the construction of the complexes $C_{\bullet}(M)$ carried out in the previous section, we may find an object $N \in \RQ'$ such that $C_{\bullet}(N)$ is constructed as $C_{\bullet}(N) = \mathrm{Cone}(\cdots \longrightarrow C_{\bullet}(N')[-1])$ with $\mathcal{D}_Q(N') = \mu_i \mathcal{D}_Q(N)$. Thus there is a morphism $C_{\bullet}(N') \longrightarrow C_{\bullet}(N)[1]$ in $\mathcal{K}^b(\RQ')$. All together we have 
 $$ C_{\bullet}(M,h) \longrightarrow C_{\bullet}(M,h') \longrightarrow C_{\bullet}(N')[-d] \longrightarrow C_{\bullet}(N)[-d+1] $$
 with $d := \sharp Z$ and $\mathcal{D}_Q(N) = \mu_Z \mathcal{D}_Q(M)$ given that 
 $$\mu_i \mathcal{D}_Q(N) = \mathcal{D}_Q(N')  = \mu_{Z'} \mathcal{D}_Q(M) = \mu_i \mu_Z \mathcal{D}_Q(M) . $$
 This concludes the proof of the claim.

  \bigskip
 
 We can now finish the proof of Theorem~\ref{thm : main thm simples}.  We set 
 $$ C_{\bullet}(h) := \mathcal{D}_Q(C_{\bullet}(M,h)) $$
 where $M$ is an object in $\RQ'$ chosen as in the claim.  In particular if $\beta = 0$ then $C_{\bullet}(h) := Y(h) \rightarrow 0 \rightarrow \cdots$ so its renormalized Euler characteristics is $1$.
 Now if $\beta \neq 0$  then by construction of $C_{\bullet}(M,h)$ we get a distinguished triangle in $\mathcal{K}^b(\HQ)$:
$$ \mathcal{D_Q}(C_{\bullet}(N,g))[-1] \longrightarrow C_{\bullet}(h) \longrightarrow \mathcal{D}_Q(C_{\bullet}(M,h')) \longrightarrow \mathcal{D_Q}(C_{\bullet}(N,g)) .  $$
  Taking the Euler characteristics we get using the induction assumption:
  $$ \chi(C_{\bullet}(h)) = [\mathcal{D}_Q(M)] F[\beta'] - [\mathcal{D}_Q(N)] F[\beta -\alpha_i] .  $$
 Now recall that  $\mathcal{D}_Q(N) = \mu_i \mathcal{D}_Q(M)$, so by Proposition~\ref{prop : mutation for objects in CQ} (cf. also  Remark~\ref{rk : Serre tiltings vs Chari's action}) we have 
 $$ [\mathcal{D}_Q(N)] = F_i \cdot  A_i^{-1} \cdot [\mathcal{D}_Q(M)]  $$
 which allows to rewrite the above identity as 
 $$ \chi^{ren}(C_{\bullet}(h))_{\mid F_i := -1} = F[\beta']  + A_i^{-1}F[\beta-\alpha_i].  $$
 Thus, up to identifying $A_i^{-1}$ with $\widehat{y}_i$ (which is standard from \cite{HL10,HL16}), we obtain the desired equality, as the relations obtained in Lemma~\ref{lem : cluster mutation} together with $F[0]=1$ determine the $F[\beta]$ uniquely. 
  \end{proof}


\end{document}